\documentclass{amsart}
\usepackage{mathtools}
\usepackage{enumerate}
\usepackage{amssymb}
\usepackage{amsmath}
\usepackage{amsthm}
\usepackage{enumitem}
\usepackage{mathrsfs}
\usepackage{hyperref}
\usepackage{todonotes}
\hypersetup{
           breaklinks=true,   
           colorlinks=true,   
        }
\numberwithin{equation}{section}
\theoremstyle{plain}
\newtheorem{thm}[equation]{Theorem}
\newtheorem{cor}[equation]{Corollary}
\newtheorem{prop}[equation]{Proposition}
\newtheorem{lem}[equation]{Lemma}
\newtheorem*{claim*}{Claim}
\theoremstyle{definition}
\newtheorem{defn}[equation]{Definition}
\newtheorem{setup}[equation]{Setup}
\newtheorem{asmp}[equation]{Assumption}
\theoremstyle{remark}
\newtheorem{remark}[equation]{Remark}
\newcommand{\delb}{\sqrt{-1}\partial\bar{\partial}}
\newcommand{\relmiddle}[1]{\mathrel{}\middle#1\mathrel{}}
\usepackage{xparse}
\usepackage{bbm}
\usepackage{comment}
\author[R.\ Murakami]{Rei Murakami}
\address{Department of Mathematics, Osaka Metropolitan University, 3-3-138, Sugimoto, Sumiyoshi-ku, Osaka, 558-8585, Japan}
\email{reimurakami66@gmail.com, reimurakami@omu.ac.jp}

\begin{document}

\title[Weak solutions of the gMA/dHYM equations: boundary cases]{Weak solutions of the generalized Monge-Ampère equation and the supercritical deformed Hermitian-Yang-Mills equation: boundary cases}

\begin{abstract}
    We prove the existence and uniqueness of weak solutions for the generalized Monge-Ampère equation and the supercritical deformed Hermitian-Yang-Mills equation in cohomology classes lying on the boundary of the solvable region. Moreover, we prove that the associated geometric flows converge to the weak solutions in the sense of currents. The proof combines viscosity-theoretic and pluripotential-theoretic techniques.
\end{abstract}

\maketitle

\section{Introduction}\label{sec:intro}
Partial differential equations play a central role in Kähler geometry. 
A basic example is the complex Monge–Ampère (MA) equation, which governs the Ricci curvature and underlies the theory of Kähler–Einstein metrics. Other important examples include the $J$-equation, introduced in \cite{Don,XChen} in connection with constant scalar curvature Kähler metrics, and the deformed Hermitian–Yang–Mills (dHYM) equation (also called the Leung–Yau–Zaslow equation), which arises in mirror symmetry \cite{LYZ}.

Recent progress shows that the solvability of a broad class of fully nonlinear elliptic equations in Kähler geometry is characterized by the existence of a $\mathcal{C}$-subsolution, which is a positivity condition on differential forms \cite{SW,FLM,DK,Sun,Sze,GS,CJY,Pin,Lina,Linb,FYZ2}. 
Moreover, this condition admits a numerical characterization in terms of the positivity of intersection numbers \cite{GChen,Song,Datar-Pingali,CLT,FM}, similarly to the Demailly-Păun characterization of Kählerness \cite{DP}. 

Given this characterization, it is natural to ask what happens when the positivity conditions fail. In \cite{DMS}, the authors conjectured the existence and uniqueness of a canonical weak solution for the $J$-equation and the dHYM equation even in the absence of positivity conditions \cite[Conjectures 1.5 and 1.12]{DMS}. Furthermore, in dimension two, they confirmed this conjecture using the weak solution theory of the MA equation \cite{BEGZ}.

In this paper, we confirm the conjecture in the boundary case (see Assumptions \ref{asmp:bdryMA} and \ref{asmp:bdrydHYM}) in a slightly weaker form (see Remrak \ref{rem:main} \ref{item:DMS}), by extending the pluripotential framework of \cite{BEGZ}. 
The corresponding existence and uniqueness result for the $J$-equation
(with $\omega$ and $\chi$ Kähler forms, unlike in Setup \ref{setup})
was announced in \cite{DMS}.
To the best of the author's knowledge, a detailed proof has not yet appeared in the literature.

Motivated by \cite{Datar-Pingali,FM}, we consider the generalized Monge-Ampère (gMA) equation \eqref{eq:gMA}, an equation that generalizes both the MA equation and the $J$-equation. Our setting is as follows. 
\begin{setup}\label{setup}
    Let $X$ be an $n$-dimensional compact Kähler manifold and $\omega$ a closed real semipositive $(1,1)$-form which is strictly positive outside a pluripolar set. Let $[\chi]$ be a nef cohomology class. Let $c_k\ge0$ be constants for $k=1,\dots,n-1$ and $c_0$ be a continuous function such that either 
    \begin{equation}\label{eq:MA}
        \sum_{k=1}^{n-1} c_k=0,\ c_0>0
    \end{equation} or 
    \begin{equation}\label{eq:gMAtop}
        \sum_{k=1}^{n-1} c_k>0,\ \int_Xc_0\omega^n\ge0
    \end{equation} 
    holds and
    \begin{equation*}
        \int_X[\chi]^n-\sum_{k=1}^{n-1}c_k[\chi]^{k}\wedge[\omega]^{n-k}=\int_Xc_0\omega^n,\quad c_0>-\varepsilon_{2.1}
    \end{equation*}
    holds, where $\varepsilon_{2.1}$ is the constant in Proposition \ref{prop:properties}. 
    Remark that the case \eqref{eq:MA} corresponds to the MA equation.
\end{setup}

The following is the assumption of the boundary case.

\begin{asmp}[boundary case]\label{asmp:bdryMA}
    Suppose that there exist sequences of Kähler classes $[\chi_i]$, Kähler forms $\omega_i$, constants $c_{k,i}\ge0$, where $k=1,\dots, n-1$, that satisfy 
        \begin{align}\label{eq:boundary}
            \chi_i\ge\chi , \quad [\chi_i]\to[\chi],\quad \omega_i\searrow \omega,\quad c_{k,i}\searrow c_k\ \text{ as }\ i\to\infty,
        \end{align}
    and that for any $p$-dimensional subvariety $V$ and any $p=1,\dots, n$, we have
        \begin{equation}\label{eq:gMAposi}
            \int_V\frac{n!}{p!}[\chi_i]^p-\sum_{k=1}^{n-1}c_{k,i}\frac{k!}{(k-n+p)!}[\chi_i]^{k-n+p}\wedge[\omega_i]^{n-k}>0.
        \end{equation}
\end{asmp}

Note that \eqref{eq:gMAposi} is equivalent to the existence of a solution $\psi_i$ of the gMA equation in $[\chi_i]$, by \cite{Datar-Pingali,FM}. The main theorem is a consequence of the analysis of a subsequential limit of $\psi_i$:

\begin{thm}\label{thm:main}
    Assume Setup \ref{setup} and Assumption \ref{asmp:bdryMA}.
    Then, there exists a unique $\chi$-psh function $\psi$ with $\sup \psi=0$ such that
    \begin{equation}\label{eq:gMA}
    \begin{cases}
        \langle\chi_\psi^n\rangle=\sum_{k=0}^{n-1}c_k\langle\chi_\psi^k\wedge\omega^{n-k}\rangle,
        \\
        \frac{n!}{p!}\langle\chi_\psi^p\rangle-\sum_{k=n-p}^{n-1}c_k\frac{k!}{k-n+p!}\langle\chi_\psi^{k-n+p}\wedge\omega^{n-k}\rangle\ge0 \ \text{for any} \ p=1,\dots n-1, 
    \end{cases}
    \end{equation}
    where $\langle\cdot\rangle$ denotes the nonpluripolar product. 
\end{thm}

\begin{remark}\label{rem:main}
    \begin{enumerate}[label={\upshape(\roman*)}]
        \item The theorem in the case \eqref{eq:MA} recovers the nef and big case of the MA equation, treated in \cite{BEGZ}, with a strictly positive continuous density. 
        In \cite{BEGZ}, more generally, the authors established the corresponding result in big classes for the MA equation with an arbitrary nonpluripolar full mass measure. 
        \item The regularity assumption on $c_0$ in our theorem is related to the use of viscosity-theoretic arguments in \ref{step:subsol} below. 
        The existence part remains valid when $c_0$ is merely bounded and lower semicontinuous (Theorem \ref{thm:existence}). Continuity is only used in the uniqueness argument through the regularization procedure in Lemma \ref{lem:app}.
        \item Under additional conditions, existence results in special cases were proved by \cite{FL,DMS,Sun23,To,GS}. Those approaches rely on global $L^\infty$ estimates, whereas our argument is based on pluripotential and viscosity methods.
        \item\label{item:DMS} For the $J$-equation, which is given by $\langle
        \chi_\psi^n\rangle=c\langle\omega\wedge\chi_\psi^{n-1}\rangle$, the conjecture in \cite[Conjecture 1.2]{DMS} only requires $\chi_\psi$ to be a Kähler current for the uniquenss.
        We do not know if a Kähler current which satisfies $\langle
        \chi_\psi^n\rangle=c\langle\omega\wedge\chi_\psi^{n-1}\rangle$ necessarily satisfies the second condition of \eqref{eq:gMA}.
    \end{enumerate}
\end{remark}

Here, we explain the strategy of the proof. Although the result covers the case of the MA equation and the proof philosophy is the same, the strategy (especially \ref{step:subsol}) is different since the right-hand side of the gMA equation varies, whereas that of the MA equation is fixed. For example, we cannot use the argument that relies on the full MA class $\mathcal{E}$. 
The proof is broken down into three steps. 

\begin{enumerate}[label={Step \arabic*}]
    \item\label{step:subsol} \textit{Subsolution.} We prove that an $L^1$-limit of $\psi_i$, solutions of approximate gMA equations, is a subsolution of \eqref{eq:gMA}, i.e. satisfies the inequality ``$\ge$''. The argument follows the strategy developed in \cite{Mur}, which was motivated by \cite[Proposition 21]{CS17}. The key is the convolution argument with the convexity of sublevel sets of the operator for the gMA equation.
    
    \item\label{step:sol} \textit{From subsolution to solution.} We observe that a subsolution is a solution. For the MA equation, the converse inequality in terms of the MA mass \cite[Proposition 1.20]{BEGZ} implies the equality, hence a solution. We establish the same type of mass inequality for degenerate $\mathcal{C}$-subsolutions (Proposition \ref{prop:superineq}). By this inequality, a subsolution is a solution.

    \item\label{step:uni} \textit{Uniqueness.} We follow the strategy of \cite[Section 3.3]{BEGZ}, which was based on the solutions of the MA equation with different measures of the same mass. We use the solutions of the gMA equation with different $c_0$ while fixing $\int_Xc_0\omega^n$. One of the key ingredients in \cite[Section 3.3]{BEGZ} was the log-concave property of the nonpluripolar product (Proposition 1.11 there). We use the convexity of the operator of the equation and the regularization arguments (Lemma \ref{lem:app}), to obtain the desired inequality (Lemma \ref{lem:BMineq}).
\end{enumerate}

We also study the weak convergence of the mixed Hessian flow in the setting of Setup \ref{setup} and Assumption \ref{asmp:bdryMA}, assuming that $\omega$ and $\chi$ are Kähler and that $c_0$ is a nonnegative constant.
The mixed Hessian flow is defined by
\begin{equation*}
    \begin{cases}
        \dot{\varphi_t}=1-\sum_{k=0}^{n-1}c_k\frac{\chi_t^k\wedge\omega^{n-k}}{\chi_t^n},\\
        \varphi_t|_{t=0}=\varphi_0,
    \end{cases}
\end{equation*}
where $\chi_t:=\chi+\delb \varphi_t>0$ and $\varphi_0$ is a $\chi$-psh function. The long-time existence of this flow is proved in \cite[Proposition 13]{CS17}.

\begin{thm}\label{thm:Hessflow}
    Let $(X,\omega)$ be an $n$-dimensional compact Kähler manifold and $\chi$ a Kähler form. Let $c_k\ge0$ be constants for $k=0,1,\dots,n-1$ satisfying the conditions in Setup \ref{setup}.  Assume that we have
    $$\int_V\frac{n!}{p!}[\chi]^p-\sum_{k=1}^{n-1}c_k\frac{k!}{(k-n+p)!}[\chi]^{k-n+p}\wedge[\omega]^{n-k}\ge0$$
    for any $p$-dimensional subvariety $V$, where $p=1,2,\dots,{n-1}$.
    Then, the mixed Hessian flow converges to the weak solution of the gMA equation in the sense of currents.
\end{thm}

\begin{remark}
    In dimension two, \cite{Mur} proved the theorem. Moreover, in the two-dimensional semipositive case \cite{FLSW} and in a certain semipositive condition in general dimension \cite{Sun23b}, they proved the same result by establishing global $L^\infty$ estimates. They further proved smooth convergence on a Zariski open subset.
\end{remark}
The proof of Theorem \ref{thm:Hessflow} follows the strategy of \cite{Mur}. That is, using the convexity of an energy functional along the flow, we first prove that the time derivative of the flow converges to zero in $L^2$. Then, the arguments from \ref{step:subsol} imply that a weak limit is a subsolution. Then, \ref{step:sol} and \ref{step:uni} imply the desired result.

With the same strategy, we establish similar results for the dHYM equation \eqref{eq:dHYM} with the supercritical phase, that is, $\theta\in(0,\pi)$. Our setup for the dHYM equation is as follows. 

\begin{setup}\label{setup2}
    Let $X$ be an $n$-dimensional compact Kähler manifold and $\omega$ a closed real semipositive $(1,1)$-form which is strictly positive outside a pluripolar set. Let $[\alpha]$ be a real $(1,1)$ cohomology class. Let $\theta\in(0,\pi)$ be a constant such that 
    \begin{equation*}
        \int_X \mathrm{Re}(\alpha+\sqrt{-1}\omega)^n-\cot\theta\mathrm{Im}(\alpha+\sqrt{-1}\omega)^n=0.
    \end{equation*}
\end{setup}
\begin{asmp}[boundary case]\label{asmp:bdrydHYM}
    Suppose that there exist sequences of real $(1,1)$ cohomology classes $[\alpha_i]$, Kähler forms $\omega_i$, and constants $\theta_i\in(0,\pi)$ that satisfy
    \begin{align}
    &\alpha_i\ge\alpha, \quad [\alpha_i]\to[\alpha],\quad \omega_i\searrow \omega,\quad \theta_i\searrow\theta\ \text{ as }\ i\to\infty,
    \end{align}
    and that the triple $(X,[\alpha_i],[\omega_i])$ is stable along a test family with a phase $\theta_i$ in the sense of \cite[Theorem 1.3]{CLT}.
\end{asmp}

By \cite[Theorem 1.3]{CLT} and \cite{Linb}, the stability of a triple $(X,[\alpha_i],[\omega_i])$ along a test family is equivalent to the existence of a solution of the twisted dHYM equation in $[\alpha_i]$. Analogously to Theorem \ref{thm:main}, we obtain the following:

\begin{thm}\label{thm:maindHYM}
    Assume Setup \ref{setup2} and Assumption \ref{asmp:bdrydHYM}. Then, there exists a unique quasi-psh function $\psi$ with $\sup\psi=0$ such that
    \begin{equation}\label{eq:dHYM}
        \begin{cases}
            \mathrm{Re}\langle(\alpha_\psi+\sqrt{-1}\omega)^n\rangle-\cot\theta\mathrm{Im}\langle(\alpha_\psi+\sqrt{-1}\omega)^n\rangle=0,\\
            \alpha_\psi\in\bar{\Gamma}_{\omega,\theta,\Theta},
        \end{cases}
    \end{equation}
    where $\langle\cdot\rangle$ denotes the nonpluripolar product, the constant $\Theta$ is any constant such that $0<\theta<\Theta<\pi$, and the definition of $\bar{\Gamma}_{\omega,\theta,\Theta}$ is given in Definition \ref{def:dhymgammabar}.
\end{thm}

\begin{remark}
        We can also establish the same result for the twisted dHYM equation, except for the uniqueness in dimension 3, since the convexity of the twisted dHYM equation with a twist of a non-constant function in dimension less than 4 is not known \cite[Proposition 5.6]{GChen} (the uniqueness in dimensions 1 and 2 follows from that of the MA equation).
\end{remark}

The dHYM flow, introduced in \cite{FYZ}, is defined by
\begin{equation*}
    \begin{cases}
        \dot{\varphi_t}=\cot\theta_\omega(\alpha_t)-\cot\theta,\\
        \varphi_t|_{t=0}=\varphi_0,
    \end{cases}
\end{equation*}
where $\alpha_t:=\alpha+\delb \varphi_t$ and the initial data $\varphi_0$ satisfies $0<\theta_\omega(\alpha_0)<\pi$. Here, the function $\theta_\omega(\alpha)$ is defined as  $\theta_\omega(\alpha):=\sum_i \operatorname{arccot}\lambda_i$, where $\lambda_i$'s are eigenvalues of $\omega^{-1}\alpha$.
The long-time existence of this flow is proved in \cite[Theorem 1.2]{FYZ} and the estimate $0<\inf\theta_\omega(\alpha_0)\le\theta_\omega(\alpha_t)\le\sup\theta_\omega(\alpha_0)<\pi$ was established in \cite[Lemma 3.2]{FYZ}.
Analogously to Theorem \ref{thm:Hessflow}, we prove the following:

\begin{thm}\label{thm:dHYMflow}
    Let $(X,\omega)$ be an $n$-dimensional compact Kähler manifold and $[\alpha]$ be a real $(1,1)$ cohomology class. Assume Setup \ref{setup2} and Assumption \ref{asmp:bdrydHYM} with $\omega_i=\omega$ and $\alpha_i=\alpha$ for any $i$. 
    Then, the dHYM flow converges to the weak solution of the dHYM equation in the sense of currents.
\end{thm}

\begin{remark}
    In dimension two, \cite{Mur} proved the theorem. Moreover, in the two-dimensional semipositive case \cite{FYZ}, they proved the same result by establishing global $L^\infty$ estimates. They further proved smooth convergence on a Zariski open subset.
\end{remark}

Although we expect our strategy to extend to more general complex Hessian equations, there are several difficulties. 
For example, the general inverse $\sigma_k$ equation was studied in \cite{Lina,Linb}. It takes the form \eqref{eq:gMA} with possibly negative coefficients $c_i$, and in particular unifies the gMA equation and the supercritical dHYM equation. 
In \cite{Lina}, the author proved that these equations have convex sublevel sets under the strict $\Upsilon$-stability condition. However, it is not known whether a property analogous to \eqref{eq:CL} holds. Such a property would imply the existence of weak solutions via the convolution argument. The uniqueness problem appears substantially more difficult, since our uniqueness argument relies on the convexity of the operator.

For the complex Hessian quotient equation $\alpha^k\wedge\omega^{n-k}=\sum_{i=0}^{k-1}c_i\alpha^i\wedge\omega^{n-i}$, although the convexity of the operator is well-known when $c_i\ge0$, we do not know if the convexity of sublevel sets is preserved when $c_0$ can be slightly negative. We also do not know whether a property similar to \eqref{eq:CL} holds. The absence of quasi-plurisubharmonicity might be a difficulty as well. 

Note that \cite{CX25a,CX25b} studied viscosity solutions for very general complex Hessian equations of the type cosidered in \cite{Sze}. Also, pluripotential solutions for the complex $m$-Hessian equation, defined by $\alpha^m\wedge\omega^{n-m}=c_0\omega^m$, have been studied extensively (see \cite{KN,PSWZ} for recent results). 

The paper is organized as follows. We study the gMA equation in Section \ref{sec:gMA} and the dHYM equation in Section \ref{sec:dHYM}. In subsection \ref{subsec:pre}, we define viscosity subsolutions. In subsection \ref{sec:existence}, we prove the existence part of Theorem \ref{thm:main} (Theorem \ref{thm:existence}). In subsection \ref{subsec:uniqueness}, we prove the uniqueness part of Theorem \ref{thm:main}. In subsection \ref{subsec:mixedHessianflow}, we prove Theorem \ref{thm:Hessflow}. Section \ref{sec:dHYM} is organized in the same way as Section \ref{sec:gMA}. The proofs in Section \ref{sec:dHYM} are mostly the same as those in Section \ref{sec:gMA}, so we only point out the differences.

\subsection*{Acknowledgements}
The author thanks Satoshi Jinnouchi for many discussions. The author also thanks Chao-Ming Lin for kindly answering his questions.

\section{The Generalized Monge-Ampère equation}\label{sec:gMA}
\subsection{Preliminaries}\label{subsec:pre}
We follow the notation in \cite{CS17}. For a vector $\lambda=(\lambda_1,\dots,\lambda_n)\in\mathbb{R}^n$, denote by 
\begin{equation*}
    S_k(\lambda):=\sum_{1\le i_1<i_2<\dots<i_k\le n}\lambda_{i_1}\lambda_{i_2}\dots\lambda_{i_k},
\end{equation*}
the $k$-th elementary symmetric function. Define $S_0=1$ and $S_{-1}=0$. For distinct indices $i_1,\dots, i_\ell$, we define
\begin{equation*}
    S_{k:i_1,\dots,i_\ell}(\lambda):=S_k(\lambda)|_{\lambda_{i_1}=\dots=\lambda_{i_\ell}=0},
\end{equation*}
while  if $i_1,\dots, i_\ell$ are not distinct, then $S_{k:i_1,\dots,i_\ell}(\lambda):=0$.
Denote the set of $n\times n$ Hermitian matrices by $\operatorname{Herm}(\mathbb{C}^n)$ and let $A\in\operatorname{Herm}(\mathbb{C}^n)$. Let $\{c_k\}_{k=1}^{n-1}$ be nonnegative constants and $c_{0,0}$ a constant. 
For $\ell=1,\dots,n-1$, we define
\begin{align*}
    P^\ell_{I_n}(A)&:=\max_{i_1,\dots,i_\ell\in\{1,\dots,n\}} \left(\sum_{k=1}^{n-1} c_k\frac{1}{{n\choose k}}\frac{S_{k-\ell;i_1,\dots i_\ell}(\lambda)}{S_{n-\ell;i_1,\dots i_\ell}(\lambda)}\right),\\
    Q_{c_{0,0},I_n}(A)&:=Q_{c_{0,0},I_n}(A) := \sum_{k=1}^{n-1}c_k\frac{1}{{n\choose k}}\frac{S_k(\lambda)}{S_n(\lambda)}+c_{0,0}\frac{1}{S_n(\lambda)},
\end{align*}
where $\lambda$ is the eigenvalue vector of $A$. We particularly denote $P^1_{I_n}$ by $P_{I_n}$.
Let us define 
\begin{equation*}
    \Gamma_{\ge0}:=\{A\in\operatorname{Herm}(\mathbb{C}^n)\mid A \text{ is semipositive}\},\quad \bar{\Gamma}:=\{A\in \Gamma_{\ge0}\mid P_{I_n}(A)\le1 \}.
\end{equation*}
The important properties of $P^\ell$ and $Q$ are the following:

\begin{prop}\label{prop:properties}
    There exists a constant $\varepsilon_{2.1}>0$, depending on $\{c_k\}_{k=1}^{n-1}$, such that if a constant $c_{0,0}$ satisfies $c_{0,0}>-\varepsilon_{2.1}$, the following hold:
    \begin{enumerate}[label={\upshape(\roman*)}]
        \item\label{item:monotonicity} $P^\ell_{I_n}$ decreases in $\Gamma_{\ge0}$ and $Q_{c_{0,0},I_n}$ decreases in $\overline{\Gamma}$. That is, if $A\ge B\in\Gamma_{\ge0}$, then $P^\ell_{I_n}(A)\le P^\ell_{I_n}(B)$, and if $A\ge B\in\bar{\Gamma}$, then $Q_{c_{0,0},I_n}(A)\le Q_{c_{0,0},I_n}(B)$.
        \item $P^\ell_{I_n}$ is convex on $\Gamma_{\ge0}$ and $Q_{c_{0,0},I_n}$ is convex on $\bar{\Gamma}$.
        \item\label{item:sublevel} The sublevel sets $\{P^\ell_{I_n}\le1\}\cap\Gamma_{\ge0}$ and $\{Q_{c_{0,0},I_n}\le1\}\cap\bar{\Gamma}$ are convex.
    \end{enumerate}
\end{prop}

\begin{proof}
    The first and second items for $P^\ell_{I_n}$ follow from \cite[Lemmas 7 and 8]{CS17}, since $c_k\ge0$. Those for $Q_{c_{0,0},I_n}$ are proved in \cite[Lemma 2.2 (3) and (5)]{Datar-Pingali}. The third item follows directly from the second item. 
\end{proof}

For an $n$-dimensional Kähler manifold $(X,\omega)$ and a closed real $(1,1)$-form $\chi$,
we define
\begin{equation}\label{eq:PQ}
    P^\ell_\omega(\chi)(x) := P^\ell_{I_n}((\omega^{-1}\chi)(x)), \quad
Q_{c_0,\omega}(\chi) (x) := Q_{c_0(x),I_n}((\omega^{-1}\chi)(x)).
\end{equation}
We particularly denote $P^1_\omega$ by $P_\omega$.
Remark that the gMA equation,
\begin{equation*}
    \chi^n=\sum_{k=0}^n c_k\chi^k\wedge\omega^{n-k},
\end{equation*}
is equivalent to $Q_{c_0,\omega}(\chi)=1$.
A Kähler metric $\chi$ is called a $\mathcal{C}$-subsolution to the gMA equation $Q_{c_0,\omega}(\chi)=1$ if $P_{\omega}(\chi)<1$, which is equivalent to
\begin{equation*}
    n\chi^{n-1}-\sum_{k=1}^{n-1}c_k k\chi^{k-1}\wedge\omega^{n-k}>0.
\end{equation*}
We give the following definition of the degenerate $\mathcal{C}$-subsolution and the viscosity subsolutions on K{\"a}hler manifolds.
The definition has appeared in many places (\cite{EGZ,CS17,DDT,CX25a,CX25b} to name a few).

\begin{defn}\label{defn:vsub}
    Let $(X,\omega)$ be a K{\"a}hler manifold. 
    We say that a closed positive $(1,1)$-current $\chi$ is a degenerate $\mathcal{C}$-subsolution (resp. a viscosity subsolution), denoted by $P_\omega(\chi)\le_v1$ (resp. $P_\omega(\chi)\le_v1$ and $Q_{c_0,\omega}(\chi)\le_v1$), if for any point $x\in X$ and any coordinate neighborhood $U$ of $x$, we have  
    $$P_\omega(\delb q)(x)\le 1\ \bigg( \text{resp. } P_\omega(\delb q)(x)\le 1 \text{ and } Q_{c_0,\omega}(\delb q)\le1\bigg),$$
    for any upper test $q$ of $\varphi$ at $x \in U$ (i.e.,\,a function $q$ is $C^2$-function defined on a neighborhood $U$ of $x$ and satisfies $q\ge\varphi$ with $q(x)=\varphi(x)$),
    where $\varphi$ is an upper semicontinuous local potential of $\chi$.
\end{defn}

\begin{remark}\label{rem:visc}
    \begin{enumerate}[label={\upshape(\roman*)}]
        \item A smooth closed real $(1,1)$-form $\chi$ is called a solution of the gMA equation if
        \begin{equation*}
            Q_{c_0,\omega}(\chi)=1,\ P_\omega(\chi)<1, \text{ and } \chi>0.
        \end{equation*}
        In this sense,  a solution is a viscosity subsolution by Proposition \ref{prop:properties} \ref{item:monotonicity}.
        \item We will see that a closed positive $(1,1)$-current $\chi$ is a degenerate $\mathcal{C}$-subsolution if and only if $\chi\in\overline{\Gamma}_\omega$, where $\overline{\Gamma}_\omega$ is defined in Definition \ref{def:gammabar}. These are also equivalent to the condition $T^p_\omega(\chi)\ge0$ for $p=1,\dots,n-1$, where $(p,p)$-current $T^p_\omega$ is defined in Definition \ref{def:llcurrent}. 
        \item\label{item:monovisc} Since any upper test $q$ of $\varphi$ at $x\in U$ as in Definition \ref{defn:vsub} is psh as in the proof of \cite[Proposition 1.3]{EGZ}, letting $\lambda_i\to\infty$, where $\lambda_i$ is an eigenvalue of $\lambda_i$ of $\delb q$ at $x\in U$, the monotonicity of $P^\ell$ implies the following: If $P_\omega(\chi)\le_v1$, then $P^\ell_\omega(\chi)\le_v1$ for any $\ell=1,\dots,n$.
    \end{enumerate}
\end{remark}

\subsection{Existence}\label{sec:existence}
In this subsection, we prove the existence part of Theorem \ref{thm:main}. More precisely, we prove the existence with $c_0$ being only bounded and lower semicontinuous, since we need this version later in the proof of uniqueness to solve the gMA equation with a lower semicontinuous function $f_1$ given in \eqref{eq:f_1}.

\begin{thm}\label{thm:existence}
    Assume Setup \ref{setup}, with $c_0$ being only bounded and lower semicontinuous, and Assumption \ref{asmp:bdryMA}.
    Then, there exists a $\chi$-psh function $\psi$ with $\sup \psi=0$ such that
    \begin{equation}\label{eq:lscgMA}
    \begin{cases}
        \langle\chi_\psi^n\rangle=\sum_{k=0}^{n-1}c_k\langle\chi_\psi^k\wedge\omega^{n-k}\rangle,
        \\
        \frac{n!}{p!}\langle\chi_\psi^p\rangle-\sum_{k=n-p}^{n-1}c_k\frac{k!}{k-n+p!}\langle\chi_\psi^{k-n+p}\wedge\omega^{n-k}\rangle\ge0 \ \text{for any} \ p=1,\dots n-1, 
    \end{cases}
    \end{equation}
    where $\langle\cdot\rangle$ denotes the nonpluripolar product.
\end{thm}

As explained in Section \ref{sec:intro}, we first construct a viscosity subsolution. 
Then, we show in Lemma \ref{lem:visctopp} that a viscosity subsolution is a pluripotential subsolution (we prove that these are equivalent in Proposition \ref{prop:viscpp} in the next subsection). Finally, we establish the converse mass inequality (Proposition \ref{prop:superineq}), which allows us to conclude that any subsolution is, in fact, a solution.

In Setup \ref{setup} and Assumption \ref{asmp:bdryMA}, by \cite{FM} (\cite{Datar-Pingali} for the projective case), 
we have a solution $\psi_i$ of the gMA equation
\begin{equation}\label{eq:appgMA}
    \begin{cases}
        \chi^n_{i,\psi_i}=\sum_{k=0}^{n-1}c_{k,i}\chi_{i,\psi_i}^k \wedge\omega_i^{n-k},\\
        n\chi_{i,\psi_i}^{n-1}-\sum_{k=1}^{n-1}c_{k,i} k\chi^{k-1}_{i,\psi_i}\wedge\omega_i^{n-k}>0 ,\quad \chi_{i,\psi_i}>0,\\
        \sup \psi_i=0,
    \end{cases}
    \end{equation}
    where $\chi_{i,\psi_i}:=\chi_i+\delb\psi_i$ and $c_{0,i}$ is a smooth function such that 
    \begin{equation}\label{eq:topnum}
        c_{0,i}>-\varepsilon_{2.1},\ \lim_{i\to\infty} c_{0,i}=c_0,\ \int_X\chi^n_i-\sum_{k=0}^{n-1}c_{k,i}\chi_i^k \wedge\omega_i^{n-k}=0.
    \end{equation}

    Hereafter, when we replace $c_k$ with $c_{k,i}$, we simply add the subindex $i$. For example, we denote by $Q_{i,c_{0,i},\omega}$ the operator defined by replacing $c_k$ with $c_{k,i}$ in the definition \eqref{eq:PQ}. 
    By weak compactness, we have an $L^1$-limit $\psi_\infty$ of $\psi_i$ with $\sup{\psi_\infty}=0$ (if we take a subsequence, which we still denote by $\psi_i$). To prove that $\chi_\infty$ is a viscosity subsolution, where $\chi_\infty=\chi+\delb\psi_\infty$, the following proposition on the limit behavior in the viscosity theory is essential:

\begin{prop}[{\cite[Proposition 4.3]{CIL}}]\label{propCIL}
    Let $\Omega \in \mathbb{C}^n$ be a domain, $v$ an upper semicontinous function on $\Omega$, $z \in \Omega$, and $q$ an upper test of $v$ at $z$. Suppose also that $u_n$ is a sequence of upper semicontinuous functions on $\Omega$ such that
    \begin{enumerate}[label={\upshape(\roman*)}]
        \item\label{itemCIL1} there exists $x_n \in \Omega$ such that $(x_n,u_n(x_n)) \rightarrow (z,v(z))$,
        \item\label{itemCIL2} if $z_n \in \Omega$ and $z_n \rightarrow x\in\Omega$, then $\limsup_{n\rightarrow\infty}u_n(z_n)\le v(x)$.
    \end{enumerate}
    Then there exist $\hat{x}_n\in \Omega$ and upper tests of $u_n$ at $\hat{x}_n$ denoted by $q_n$ such that
    $$(\hat{x}_n,u_n(\hat{x}_n),Dq_n(\hat{x}_n),D^2q_n(\hat{x}_n))\rightarrow(z,v(z),Dq(z),D^2q(z)).$$
\end{prop}

\begin{prop}\label{prop:vsub}
    We have $P_\omega(\chi_\infty)\le_v 1$ and $Q_{c_0,\omega}(\chi_\infty)\le_v 1$ in $\{\omega>0\}$.
\end{prop}

\begin{proof}[Proof of Proposition \ref{prop:vsub}]
    Following the proof of \cite[Proposition 21]{CS17}, we first prove $P_\omega(\chi_\infty)\le_v1$.
    Take a coordinate neighborhood $U\subset\{\omega>0\}$ of $p\in\{\omega>0\}$. Let $\omega_0$ be a K{\"a}hler form on some neighborhood $U_0 \subset U$ with constant coefficients such that $\omega_0 \le \omega$. Denote by $F_i$ an upper semicontinuous local potential of $\chi_i$ on $U_0$. Take a standard mollifier $\rho$, i.e.,\,a function $\rho$ satisfies $$\rho\ge0, \ \rho(w)=0 \text{ for $|w|>1$},\ \text{and } \int\rho(w)dw=1.$$ Define $u_{i}=F_i+\psi_i$ and denote by $u_{i,\delta}$ the regularization of $u_i$, i.e.,
    \begin{equation*}
        u_{i,\delta}(z)=\int_{\mathbb{C}^n} u_i(z-\delta w)\rho(w)dw.
    \end{equation*}
    The function $u_{i,\delta}$ is defined on $U_{0,\delta} =\{z\in U_0\mid d(z,\partial U_0)> \delta \}.$
    Note that we have
    \begin{equation*}
        \begin{split}
        P_{i,\omega_0}(\delb u_{i,\delta})(z)
        &=P_{i,I_n}\left(\omega^{-1}_0(z) \int_{\mathbb{C}^n}\delb u_i(z-\delta w)\rho(w)dw
        \right)\\
        &=P_{i,I_n}\left( \int_{\mathbb{C}^n}(\omega^{-1}_0 \chi_{i,\psi_i})(z-\delta w)\rho(w)dw
        \right)\\
        &\le P_{i,I_n}\left( \int_{\mathbb{C}^n}(\omega^{-1}_i \chi_{i,\psi_i})(z-\delta w)\rho(w)dw
        \right)\\
        &\le 1,
        \end{split}
    \end{equation*}
    In the second equality, we used the fact that $\omega_0$ has constant coefficients. In the first inequality, we used that $\omega_0\le \omega\le\omega_i$ and the monotonicity of $P_{i,I_n}$ as in Proposition \ref{prop:properties} \ref{item:monotonicity}. In the second inequality, we used the convexity of the sublevel set of $P_{i,I_n}$ as in Proposition \ref{prop:properties} \ref{item:sublevel}. We claim that $P_{\omega_0}(\delb u_{\infty,\delta})\le_v 1$ on $U_{0,\delta}$.
    Indeed, note that $u_{i,\delta}$ converges to $u_{\infty,\delta}$ uniformly in $U_{0,\delta}$. Therefore, a sequence $\{ u_{i,\delta}\}$ and the limit $u_{\infty,\delta}$ satisfy the conditions in Proposition \ref{propCIL}. Choose a point $z \in U_{0,\delta}$ and an upper test $q$ of $u_{\infty,\delta}$ at $z$. By Proposition \ref{propCIL}, there exist points $\hat{x}_i$ and upper tests $q_i$ of $u_{i,\delta}$ at them such that the assertion satisfies. It implies 
    $$P_{i,\omega_0}(\delb q_i)(\hat{x}_i) \rightarrow P_{\omega_0}(\delb q)(z) \le 1,$$
    which yields $P_{\omega_0}(\delb u_{\infty,\delta})\le_v 1$ on $U_{0,\delta}$.
    Next, we let $\delta$ tend to zero. Since $u_\infty = F + \psi_\infty$ is subharmonic, where $F$ is an upper semicontinuous local potential of $\chi$ on $U_0$, a sequence $\{ u_{\infty,\delta} \}$ is decreasing. This property implies that the sequence satisfies the conditions in Proposition \ref{propCIL}. 
    Indeed, condition \ref{itemCIL1} is obvious by taking the same point as a sequence. For condition \ref{itemCIL2}, note that for any fixed $\delta_0>0$, the function $u_{\infty,\delta_0}$ is upper semicontinuous, so for any $\delta < \delta_0$ and $z_\delta \in V \subset\subset U_0$ satisfying $z_\delta \rightarrow x \in V$,
    $$\limsup{u_{\infty,\delta}}(z_\delta)\le \limsup{u_{\infty,\delta_0}}(z_\delta) \le u_{\infty,\delta_0}(x).$$
    Letting $\delta_0$ tend to zero, we get condition \ref{itemCIL2}. Therefore we can apply Proposition \ref{propCIL} again to get
    $$P_{\omega_0}(\delb u_\infty) \le_v 1 \ \text{on $U_0$}.$$
    Especially at point $p$, for any upper test $q$ of $u_\infty$ at $p$, we have
    $$P_{\omega_0}(\delb q)(p) \le 1.$$
    By taking a sequence of K{\"a}hler forms $\{\omega_{0,j}\}$ with constant coefficients such that $\omega_{0,j} \le \omega$ and $\omega_{0,j}(p) \rightarrow \omega(p)$ (by taking $U_0$ smaller and smaller), we get
    $$P_\omega(\delb q)(p) \le 1.$$
    
    Next, we prove $Q_\omega(\chi_\infty)\le_v1$ in the same way. Since we have $P_{i,\omega_0}(\delb u_i)\le1$, by the monotonicity of $P_{i,\omega_0}$ on $\Gamma_{\ge0}$ and the monotonicity of $Q_{i,c_{0,0},I_n}$ on $\bar{\Gamma}$,  
    we see that for $z \in U_{0,\delta}$,
    \begin{equation}\label{eq:Qconv}
        \begin{split}
        Q_{i,c_{0,0},\omega_0}(\delb u_{i,\delta})(z)
        &=Q_{i,c_{0,0},I_n}\left(\omega^{-1}_0(z) \int_{\mathbb{C}^n}\delb u_i(z-\delta w)\rho(w)dw
        \right)\\
        &=Q_{i,c_{0,0},I_n}\left( \int_{\mathbb{C}^n}(\omega^{-1}_0 \chi_{i,\psi_i})(z-\delta w)\rho(w)dw
        \right)\\
        &\le Q_{i,c_{0,i},I_n}\left( \int_{\mathbb{C}^n}(\omega^{-1}_0 \chi_{i,\psi_i})(z-\delta w)\rho(w)dw
        \right)\\
        &\le Q_{i,c_{0,i},I_n}\left( \int_{\mathbb{C}^n}(\omega^{-1}_i \chi_{i,\psi_i})(z-\delta w)\rho(w)dw
        \right)\\
        &\le 1,
        \end{split}
    \end{equation}
    where $c_{0,0}$ is a constant such that $c_{0,i}\ge c_{0,0}>-\varepsilon_{2.1}$ for any large $i$ on $U_0$.
    In the first inequality, we used that $c_{0,i}\ge c_{0,0}$ and $\omega_0^{-1}\cdot\chi_{i,\psi_i}$ is positive. In the second inequality, we used $P_{i,I_n}(\omega_0^{-1}\cdot\chi_{i,\psi_i})\le P_{i,I_n}(\omega_i^{-1}\cdot\chi_{i,\psi_i})<1$ and the monotonicity of $Q_{i,c_{0,i},I_n}$ in $\bar{\Gamma}_i$. In the last inequality, we used the convexity of the sublevel set of $Q_i$ in $\bar{\Gamma}_i$ as in Proposition \ref{prop:properties} \ref{item:sublevel}. By the same argument as above, we have $Q_{c_{0,0},\omega}(\chi_\infty)(p)\le_v 1$. Finally, since $\liminf_{i\to\infty}\liminf_{z\to p} c_{0,i}(z)=c_0(p)$ by lower semicontinuity and the convergence $c_{0,i}\to c_0$, we obtain $Q_{c_0,\omega}(\chi_\infty)(p)\le_v1$. 
\end{proof}

We note that the arguments in the proof clarify that the condition on the positivity of currents related to \cite[Definition 3.3]{GChen} and \cite[Definition 1]{Datar-Pingali} implies the condition of the degenerate $\mathcal{C}$-subsolution (Definition \ref{defn:vsub}). Namely, if a closed positive $(1,1)$-current $\chi_\varphi$ is in $\bar{\Gamma}^\ell_\omega$, defined below, then $P^\ell_\omega(\chi_\varphi)\le_v 1$. Similarly, we can see that if a closed positive $(1,1)$-current $\chi_\varphi$ is in $\bar{\Gamma}'_\omega$, defined below, then $Q_{c_0,\omega}(\chi)\le_v 1$ in addition to $P_{\omega}(\chi_\varphi)\le_v1$.
Here, the sets $\bar{\Gamma}^\ell_\omega$ and $\bar{\Gamma}'_\omega$ are defined as follows:
\begin{defn}\label{def:gammabar}
    Let $\varphi$ be a $\chi$-psh function.
    \begin{enumerate}[label={\upshape(\roman*)}]
        \item\label{item:gammal} A closed real positive $(1,1)$-current $\chi_\varphi$ is in $\bar{\Gamma}^\ell_\omega$ if for any open subset $U_0$ in any coordinate neighborhood and any K{\"a}hler form $\omega_0$ on $U_0$ with constant coefficients such that $0<\omega_0 \le \omega$, we have
        \begin{equation}\label{eq:conv}
            P^\ell_{\omega_0}(\delb u_\delta) \le 1
        \end{equation} 
        for any $\delta>0$, where $u$ is an upper semicontinuous local potential of $\chi_\varphi$ and $u_\delta$ is its regularization via convolution. We particularly denote $\bar{\Gamma}^1_\omega$ by $\bar{\Gamma}_\omega$.
        \item\label{item:gamma'} A closed real positive $(1,1)$-current $\chi_\varphi$ is in $\bar{\Gamma}'_\omega$ if $\chi_\varphi\in\bar{\Gamma}_\omega$ and 
        $$Q_{c_{0,0},\omega_0}(\delb u_\delta) \le 1$$ 
        where $c_{0,0}$ is a constant which satisfies $c_0>c_{0,0}$ in $U_0$ and the other notations are the same as \eqref{eq:conv}.
    \end{enumerate}
\end{defn}

Using the argument in \cite[Proposition 2.11]{Mur}, which is the case for the $J$-equation and $\ell=1$, we can prove that these conditions $\chi_\varphi\in\bar{\Gamma}^\ell_\omega$ and $P^\ell_\omega(\chi_\varphi)\le_v 1$ are equivalent. Similarly, we can prove that conditions ``$\chi_\varphi\in\bar{\Gamma}'_\omega$'' and ``$Q_{c_0,\omega}(\chi_\varphi)\le_v1$ and $P_\omega(\chi_\varphi)\le_v 1$'' are equivalent.
Since we use the same argument to prove Lemma \ref{lem:visctopp} below, we record it here.
    \begin{prop}\label{prop:equivalence}
    Let $\varphi$ be a $\chi$-psh function.
        \begin{enumerate}[label={\upshape(\roman*)}]
            \item\label{item:eqgammal} A closed positive  $(1,1)$-current $\chi_\varphi$ satisfies $P^\ell_\omega(\chi_\varphi)\le_v 1$ in $\{\omega>0\}$ if and only if $\chi_\varphi \in \bar{\Gamma}^\ell_\omega$ in $\{\omega>0\}$.
            \item\label{item:eqgamma'} Suppose that $P_\omega(\chi_\varphi)\le_v1$. Then, a closed positive  $(1,1)$-current $\chi_\varphi$ satisfies $Q_{c_0,\omega}(\chi_\varphi)\le_v 1$ in $\{\omega>0\}$ if and only if $\chi_\varphi \in \bar{\Gamma}'_\omega$ in $\{\omega>0\}$.
        \end{enumerate}
    \end{prop}
    \begin{proof}
        We first prove \ref{item:eqgammal} following the arguments in \cite[Proposition 1.11 and Theorem 1.9]{EGZ}. Since the ``if'' part is proved by the arguments above, we prove the ``only if'' part. Assume $P^\ell_\omega(\chi_\varphi)\le_v 1$ in $\{\omega>0\}$.
        Take an open subset $U_0$ in a coordinate neighborhood in $\{\omega>0\}$ and a K{\"a}hler form $\omega_0$ with constant coefficients on $U_0$ such that $0<\omega_0\le\omega$. Note that an upper test of a psh function is psh as in the proof of \cite[Proposition 1.3]{EGZ}. Thus, we have $P^\ell_{\omega_0}(\chi_\varphi)\le_v 1$ on $U_0$. Let $u$ and $v$ be upper semicontinuous local potentials of $\chi_\varphi$ and $\omega$ respectively. Define $u^C:=\sup\{u,av-C\}$, where $C$ is a positive number and $a$ is a real number large enough to satisfy $P^\ell_\omega(a\omega)\le 1$. Since the supremum of a finite number of viscosity subsolutions is a viscosity subsolution, we have $P^\ell_{\omega_0}(
        \delb u^C)\le_v 1$ on $U_0$. Denote by $u^{C,\varepsilon}$ the sup convolution of $u^C$, i.e.
        $$u^{C,\varepsilon}(x) = \sup \left\{ u^C(y)-\frac{1}{2\varepsilon^2}|x-y|^2 \, \relmiddle{|} y \in U_0 \right\}$$
        for $x\in (U_0)_{A_C \varepsilon}$, where $(U_0)_{A_C \varepsilon}:=\{x\in U_0\mid \mathrm{dist}(x,\partial U_0)>A_C \varepsilon\}$ and $A_C$ is a real number such that $A_C^2>2\, \mathrm{osc}_{U_0} \, u^C$ which is well defined since $u^C$ is bounded. By this definition of $A_C$, the supremum is attained in $B(x, A_C \varepsilon)$.
        Since $\omega_0$ has constant coefficients, from the proof of \cite[Proposition 4.2]{Ishii}, we see that a function $u^{C,\varepsilon}$ satisfies $P^\ell_{\omega_0}(\delb u^{C,\varepsilon})\le_v 1$ on $(U_0)_{A_C\varepsilon}$.
        Note that $u^{C,\varepsilon}$ is semiconvex, since the supremum of semiconvex functions is semiconvex. By Aleksandrov's theorem, $u^{C,\varepsilon}$ is twice differentiable almost everywhere. Moreover, the singular part of its distributional Hessian is nonnegative. Thus, by the convexity of the sublevel sets of $P^\ell_{\omega_0}$ as in Proposition \ref{prop:properties} \ref{item:sublevel} and the monotonicity of $P^\ell_{\omega_0}$, for $z \in (U_0)_{A_C\varepsilon+\delta}$, we have
        \begin{equation}\label{eq:supconv}
        \begin{aligned}
        P^\ell_{\omega_0}(\delb(u^{C,\varepsilon})_\delta) (z)&\le P^\ell_{\omega_0}\left(\int_{B_1} (\delb u^{C,\varepsilon})(z-\delta w)\rho(w)dw\right)\\
        &\le 1.
        \end{aligned}
        \end{equation}
        Since a sequence $\{u^{C,\varepsilon}\}$ decreases to $u^C$, the Lebesgue theorem implies that $u^{C,\varepsilon}$ converges to $u^C$ in $L^1$. Therefore, a sequence $\{(u^{C,\varepsilon})_\delta\}$ decreases to $(u^C)_\delta$. By the same arguments as above, Proposition \ref{propCIL} confirms $P^\ell_{\omega_0}(\delb(u^C)_\delta) \le 1$ in $(U_0)_\delta$. By applying the same arguments to a sequence $\{(u^C)_\delta\}_C$, as $C$ tends to $\infty$, we get $P^\ell_{\omega_0}(\delb u_\delta)\le 1$, which means $\chi_\varphi\in \overline{\Gamma}_\omega$.

        To prove \ref{item:eqgamma'}, since $c_0$ is bounded, note that we can choose $a$ large enough to satisfy $Q_{c_0,\omega}(a\omega)\le1$ in addition to $P_\omega(a\omega)\le1$. Thus, using the convexity of sublevel sets of $Q_{c_0,\omega}$ as in Proposition \ref{prop:properties} \ref{item:sublevel} and the monotonicity of $Q_{c_0,\omega}$, the same argument yields the desired result.
    \end{proof}

Next, we prove that a viscosity subsolution is a pluripotential subsolution (Lemma \ref{lem:visctopp}). Later in Proposition \ref{prop:viscpp}, we prove that they are equivalent.

\begin{defn}\label{def:llcurrent}
    For a $\chi$-psh function $\varphi$, we define a $(p,p)$-current $T^p_\omega(\chi_\varphi)$ by
    \begin{align*}
        T^p_\omega(\chi_\varphi)&:=\frac{n!}{p!}\langle\chi_\varphi^p\rangle-\sum_{k=0}^{n-1}c_k\frac{k!}{(k-n+p)!}\langle\chi_\varphi^{k-n+p}\wedge\omega^{n-k}\rangle,\\
    \end{align*}
    where $\langle\cdot\rangle$ denotes the nonpluripolar product.
    We also denote by $T^p_{i,\omega}$ the current defined by replacing $c_k$ with $c_{k,i}$ in the definitions of $T^p_\omega$ above.
\end{defn}

\begin{lem}\label{lem:visctopp}
    Let $\varphi$ be a $\chi$-psh function.
    \begin{enumerate}[label={\upshape(\roman*)}]
        \item\label{item:visctoppn-1} 
        If $P^\ell_{\omega}(\chi_\varphi)\le_v 1$ in $\{\omega>0\}$, then $T^p_\omega(\chi_\varphi)\ge0$ in $X$ for all $p\le n-\ell$.
        \item\label{item:visctoppn} 
        Suppose $P_{\omega}(\chi_\varphi)\le_v 1$. If $Q_{c_0,\omega}(\chi_\varphi)\le_v 1$ in $\{\omega>0\}$, then $T^n_\omega(\chi_\varphi)\ge0$ in $X$.
    \end{enumerate}
\end{lem}

\begin{proof}
    The strategy follows that of \cite[Proposition 1.11, Theorem 1.9, and Corollary 2.6]{EGZ}. The key is the locality of the non-pluripolar product in the plurifine topology. Due to this property, it sufficies to consider the local version with bounded psh functions. This point was overlooked in the proof of \cite[Lemma 2.13]{Mur}.
    
    First, we prove \ref{item:visctoppn-1}. Suppose that a $\chi$-psh function $\varphi$ satisfies $P^\ell_\omega(\chi_\varphi)\le_v 1$ in $\{\omega>0\}$. With the same notation as in the proof of Proposition \ref{prop:equivalence}, the proof of Proposition \ref{prop:equivalence} concludes that we have $P^\ell_{\omega_0}(\delb (u^C)_\delta)\le1$ which is equivalent to $T^{n-\ell}_{\omega_0}(\delb (u^C)_\delta)\ge0$ in $(U_0)_\delta$. Since $u^C$ is bounded, letting $\delta\to0$, the inequality converges to $T^{n-\ell}_{\omega_0}(\delb u^C)\ge0$ on $U_0$ by \cite[Theorem 2.1]{BT}. In particular, by the locality in the plurifine topology of the nonpluripolar product, we have 
    \begin{equation*}
        \mathbbm{1}_{\{u>av-C\}}T_{\omega_0}^{n-\ell}(\delb u)=\mathbbm{1}_{\{u>av-C\}} T^{n-\ell}_{\omega_0}(\delb u^C)\ge0
    \end{equation*}
    on $U_0$. Letting $C\to\infty$, by the definition of the nonpluripolar product, we obtain $T^{n-\ell}_{\omega_0}(\chi_\varphi)\ge0$ on $U_0$.
    By approximating $\omega$ from below with Kähler forms with coefficients being simple functions $\sum_j\mathbbm{1}_{U_{0,j}}\omega_{0,j}$, we further obtain $T^{n-\ell}_\omega(\chi_\varphi)\ge0$ in $\{\omega>0\}$. Then, since the nonpluripolar product puts no mass on $X\setminus
    \{\omega>0\}$, which is a pluripolar set by assumption, we conclude that $T^{n-\ell}_\omega(\chi_\varphi)\ge0$ in $X$. Since $P^\ell_\omega(\chi_\varphi)\le_v1$ implies $P^j_\omega(\chi_\varphi)\le_v1$ for any $j\ge \ell$ as in Remark \ref{rem:visc} \ref{item:monovisc}, the same argument implies $T^p_\omega(\chi_\varphi)\ge0$ in $X$ for $p\le n-\ell$.

    The proof of \ref{item:visctoppn} is completely the same as above, using the arguments in the proof of Proposition \ref{prop:equivalence} \ref{item:eqgamma'}. 
\end{proof}

Finally, the following inequality implies that a pluripotential subsolution is indeed a solution.

\begin{prop}\label{prop:superineq}
    If a $\chi$-psh function $\varphi$ satisfies $P_\omega(\chi_\varphi)\le_v 1$ in $\{\omega>0\}$, then we have
    $$\int_X T^n_\omega(\chi_\varphi)\le0.$$
\end{prop}

\begin{proof}
    Define $\varphi_{i,s}:=(1-s)\varphi+s\psi_i$, where $\psi_i$ is a solution of \eqref{eq:appgMA}. Note that since $\chi\le\chi_i, \ c_{k,i}\le c_k$ and an upper test of a psh function is psh as in the proof of \cite[Proposition 1.3]{EGZ}, we have $P_{i,\omega}(\chi_{i,\varphi})\le_v 1$. On the other hand, we have $P_{i,\omega}(\chi_{i,\psi_i})\le 1$ by \eqref{eq:appgMA}. Thus, by the convexity of the sublevel sets of $P_{i,\omega}$, we obtain $P_{i,\omega}(\chi_{i,\varphi_{i,s}})\le_v 1$ (see, e.g., the proof of \cite[Theorem 5.8]{CC}). By Proposition \ref{prop:equivalence}, we see that $T^{n-1}_{i,\omega}(\chi_{i,\varphi_{i,s}})$ is positive. Using the multilinearity of the nonpluripolar product and Lemma \ref{lem:ppineq} below, we have the inequality
    \begin{equation*}
        \begin{split}
            \frac{d}{ds}\int_X T^n_{i,\omega}(\chi_{i,\varphi_{i,s}})
            =\int_X \Big\langle(\chi_{i,\psi_i}-\chi_{i,\varphi})\wedge T^{n-1}_{i,\omega}(\chi_{i,\varphi_{i,s}})\Big\rangle\ge 0.
        \end{split}
    \end{equation*}
    Consequently, we obtain
    $$\int_X T^n_{i,\omega}(\chi_{i,\varphi})\le
    \int_X T^n_{i,\omega}(\chi_{i,\psi_i})=\int_X\chi_i^n-\sum_{k=1}^n c_{k,i}\chi_i^k\wedge\omega^{n-k}.$$
    Note that the right-hand side converges to zero by \eqref{eq:topnum} and \eqref{eq:boundary}. Thus, letting $i\to\infty$, the multilinearity of the nonpluripolar product implies the assertion.
\end{proof}

\begin{lem}\label{lem:ppineq}
    Let $\psi$ be a $\chi_i$-psh function such that $T^{n-1}_{i,\omega}(\chi_{i,\psi})$ is positive. If $\chi_i$-psh functions $\varphi_1,\varphi_2$ satisfy $\varphi_1\le\varphi_2+C$ for some constant $C$, then we have
    $$\int_X\left\langle\chi_{i,\varphi_1}\wedge T^{n-1}_{i,\omega}(\chi_{i,\psi})\right\rangle\le\int_X\left\langle\chi_{i,\varphi_2}\wedge T^{n-1}_{i,\omega}(\chi_{i,\psi})\right\rangle.$$
\end{lem}

\begin{proof}
    We follow the arguments by \cite[Section 2]{DDL}. Define $\varphi_{1,2,k}:=\sup\{\varphi_1,\varphi_2-k\}$. Since $\varphi_{1,2,k}-\varphi_2$ is globally bounded, by \cite[Proposition 2.1]{DDL}, we see that
    $$\int_X\left\langle\chi_{i,\varphi_{1,2,k}}\wedge T^{n-1}_{i,\omega}(\chi_{i,\psi})\right\rangle
    =\int_X\left\langle\chi_{i,\varphi_2}\wedge T^{n-1}_{i,\omega}(\chi_{i,\psi})\right\rangle.$$
    Fix $C>0$ and $\varepsilon>0$. Let us define
    \begin{align*}
        &f_{\varphi,k}^{C,\varepsilon}:=\frac{\sup\{\varphi_{1,2,k}-\psi_i+C,0\}}{\sup\{\varphi_{1,2,k}-\psi_i+C,0\}+\varepsilon}, \quad f_\varphi^{C,\varepsilon}:=\frac{\sup\{\varphi_1-\psi_i+C,0\}}{\sup\{\varphi_1-\psi_i+C,0\}+\varepsilon},\\
        &f_{\psi}^{C,\varepsilon}:=\frac{\sup\{\psi-\psi_i+C,0\}}{\sup\{\psi-\psi_i+C,0\}+\varepsilon}, 
    \end{align*}
    and $$\varphi_{1,2,k}^C:=\sup\{\varphi_{1,2,k},\psi_i-C\},\quad
    \varphi_1^C:=\sup\{\varphi_1,\psi_i-C\}, \quad\psi^C:=\sup\{\psi,\psi_i-C\}.$$  
    By the locality of the nonpluripolar product, we have
    $$f_{k}^{C,\varepsilon}\left\langle\chi_{i,\varphi_{1,2,k}}\wedge T^{n-1}_{i,\omega}(\chi_{i,\psi})\right\rangle=f_{k}^{C,\varepsilon}\left\langle\chi_{i,\varphi_{1,2,k}^C}\wedge T^{n-1}_{i,\omega}(\chi_{i,\psi^C})\right\rangle,$$
    where $f_{k}^{C,\varepsilon}=f_{\varphi,k}^{C,\varepsilon}f_{\psi}^{C,\varepsilon}$. Since $\varphi_{1}^C$ is bounded, we have
    $$f_{k}^{C,\varepsilon}\left\langle\chi_{i,\varphi_{1,2,k}^C}\wedge T^{n-1}_{i,\omega}(\chi_{i,\psi^C})\right\rangle\to f^{C,\varepsilon}\left\langle\chi_{i,\varphi_{1}^C}\wedge T^{n-1}_{i,\omega}(\chi_{i,\psi^C})\right\rangle$$
    as $k\to\infty$, where $f^{C,\varepsilon}:=f_{\varphi}^{C,\varepsilon}f_\psi^{C,\varepsilon}$. In particular, we have 
    \begin{equation*}
    \begin{split}
    &\varliminf_{k\to\infty}\int_X\left\langle\chi_{i,\varphi_{1,2,k}}\wedge T^{n-1}_{i,\omega}(\chi_{i,\psi})\right\rangle\\
    =&\varliminf_{k\to\infty}\lim_{C\to\infty}\int_X \mathbbm{1}_{\{\varphi_{1,2,k}^C >\psi_i-C\}\cap\{\psi>\psi_i-C\}}\left\langle\chi_{i,\varphi_{1,2,k}^C}\wedge T^{n-1}_{i,\omega}(\chi_{i,\psi^C})\right\rangle\\
    \ge&
    \varliminf_{k\to\infty}\int_X f_{k}^{C,\varepsilon}\left\langle\chi_{i,\varphi_{1,2,k}^C}\wedge T^{n-1}_{i,\omega}(\chi_{i,\psi^C})\right\rangle
    \ge
    \int_Xf^{C,\varepsilon}\left\langle\chi_{i,\varphi_{1}^C}\wedge T^{n-1}_{i,\omega}(\chi_{i,\psi^C})\right\rangle,
    \end{split}
    \end{equation*}
    where the first inequality follows from the positivity of $T^{n-1}_{i,\omega}(\chi_{i,\psi^C})$ and the inequality $1\ge f_k^{C,\varepsilon}$ on $\{\varphi_{1,2,k}^C >\psi_i-C\}\cap\{\psi>\psi_i-C\}$. 
    Letting $\varepsilon$ approach zero and $C$ approach $\infty$, we obtain the desired result.
\end{proof}

We summarize the proof of the existence part of Theorem \ref{thm:main} here.

\begin{proof}[Proof of Theorem \ref{thm:existence}]
    By Setup \ref{setup} and Assumption \ref{asmp:bdryMA}, we have the solutions $\psi_i$ of approximate gMA equations \eqref{eq:appgMA}. By taking a subsequence, we obtain a weak limit $\psi_\infty$ of $\{\psi_i\}$. By Proposition \ref{prop:vsub} and Lemma \ref{lem:visctopp}, a closed positive current $\chi_\infty:=\chi+\delb\psi_\infty$ satisfies $P_\omega(\chi_\infty)\le_v1$ in $\{\omega>0\}$ and $T^n_{\omega}(\chi_\infty)\ge0$ in $X$. 
    By Proposition \ref{prop:superineq}, the total mass of $T^n_\omega(\chi_\infty)$ equals zero. Thus, we obtain $T^n_\omega(\chi_\infty)=0$ and $T^p_\omega(\chi_\infty)\ge0$ for $p=1,\dots, n-1$, which are \eqref{eq:lscgMA}. 
\end{proof}

\subsection{Uniqueness}\label{subsec:uniqueness}

In this subsection, we prove the uniqueness part of Theorem \ref{thm:main}. The arguments are based on the Monge-Ampère case \cite[Section 3]{BEGZ}, \cite{Dinew}.

First, we establish the equivalence between viscosity and pluripotential subsolutions (Proposition \ref{prop:viscpp}). With this equivalence, we can use the viscosity theoretic procedure (e.g. taking the supremum or the convex combination) without violating the property of pluripotential subsolutions. This fact will be used in the proof of uniqueness. We do not know whether the fact can be verified directly without using the viscosity method (cf. \cite[Corollary 1.10]{GZ} for the MA equation). 

To prove equivalence, the key is the local smooth approximation property (Lemma \ref{lem:app} below), which leads to the maximum principle (Corollary \ref{cor:comp}). After establishing the maximum principle,  the proof of equivalence is mostly the same as that of \cite[Propositions 1.5 and 1.11]{EGZ}.
The second item of Lemma \ref{lem:app} below will be used only later in the proof of Lemma \ref{lem:BMineq}, a key inequality for uniqueness.

\begin{lem}\label{lem:app}
    Let $\Omega$ and $\Omega'$ be bounded domains in $\mathbb{C}^n$ such that $\Omega'\subset\subset\Omega$, and $\omega$ be a Kähler form in $\Omega$. \begin{enumerate}[label={\upshape(\roman*)}]
        \item\label{item:appP} Suppose that $\varphi$ is a bounded psh function in $\Omega$ such that $P^\ell_{\omega}(\delb \varphi)\le_v1$ in $\Omega$. Then, for any $d>1$, there exists a decreasing sequence of smooth psh functions $\{\varphi_a\}_{a\in\mathbb{N}}$ in $\Omega'$ such that $\varphi_a\searrow\varphi$ and $P^\ell_\omega(\delb \varphi_a)\le d$ in $\Omega'$.
        \item\label{item:appQ} Suppose that $\varphi$ is a bounded psh function in $\Omega$ such that $P_{\omega}(\delb \varphi)\le_v1$ and $Q_{c_0,\omega}(\delb \varphi)\le_v1$ in $\Omega$, where $c_0>-\varepsilon_{2.1}$ is a continuous function. Then, for any $d>1$, there exists a decreasing sequence of smooth psh functions $\{\varphi_a\}_{a\in\mathbb{N}}$ in $\Omega'$ such that $\varphi_a\searrow\varphi$, $P_\omega(\delb \varphi_a)\le d$, and $Q_{c_0,\omega}(\delb \varphi_a)\le_v d$ in $\Omega'$.
    \end{enumerate}
\end{lem}

\begin{proof}
   First, we prove \ref{item:appP}. This essentially follows from the simple case of the arguments of \cite[Proposition 4.1]{Datar-Pingali} and \cite[Section 4]{GChen}, based on \cite{BK}. Define $$\psi_\varepsilon:=\varphi+\varepsilon(\vert z\vert^2+C),$$
   where $C$ is a large constant defined later. Then, we see that $\psi_\varepsilon\in\operatorname{PSH}(-\varepsilon\omega_\mathrm{Euc})$ for any $0<\varepsilon<1$, and $(-\varepsilon\omega_{\mathrm{Euc}})_{\psi_\varepsilon}\in\bar{\Gamma}^\ell_\omega$ by Proposition \ref{prop:equivalence}, where $\omega_\mathrm{Euc}:=\sum \sqrt{-1} dz^i\wedge d\bar{z}^i$.
   As in the proof of \cite[Proposition 4.1]{Datar-Pingali}, we cover $\Omega'$ with a finite number of coordinate balls $B^i_r$ centered at $x_i$ of radius $r\ll1$ (with respect to $\omega_\mathrm{Euc}$) so small that we have 
   \begin{equation}\label{eq:constomega}
       \omega_0^i\le\omega\le(1+\varepsilon)\omega_0^i\ \text { on } B^i_{2r}:=B_{2r}(x_i),
   \end{equation}
   where $\omega_0^i$ is a Kähler form on $B^i_{2r}$ with constant coefficients. Note that we have $\delb\vert z-x_i\vert^2=\omega_\mathrm{Euc}$ on $B^i_{2r}$, that is, the notation in \cite{Datar-Pingali} corresponds to $\varphi^i_0=\vert z-x_i\vert^2$. In addition, the function $\psi_\varepsilon-\varepsilon\vert z-x_i\vert^2$ is bounded, since $\varphi$ is bounded, and decreases by choosing $C\ge 4r^2$ as $\varepsilon\searrow0$. 
   Therefore, adapting the arguments of \cite[Proposition 4.1]{Datar-Pingali} (especially \cite[Proposition 4.1 (3)]{GChen}), we obtain a sequence of smooth psh functions $\{\Psi_{\varepsilon,\delta}\}_\delta$, given by the regularized maximum of 
   \begin{equation}\label{eq:psidelta}
       (\psi_\varepsilon-\varepsilon\vert z-x_i\vert^2)_{\delta}
   \end{equation}
   running $i$, where $(\cdot)_{\delta}$ denotes regularization by convolution as in the proof of Proposition \ref{prop:vsub}. Note that the function in \eqref{eq:psidelta} is in $\bar{\Gamma}_{\omega_0}^\ell$. By the arguments of the regularized maximum in \cite[p561--p562]{GChen}, the functions $\Psi_{\varepsilon,\delta}$ satisfy
   \begin{equation}\label{eq:appPineq}
       \begin{aligned}
           \frac{n!}{p!}(\delb\Psi_{\varepsilon,\delta})^p-\sum_{k=1}^{n-1}\frac{c_k}{(1+\varepsilon)^{n-k}}\frac{k!}{(k-n+p)!}(\delb\Psi_{\varepsilon,\delta})^{k-n+p}\wedge\omega^{n-k}\ge0,
       \end{aligned}
   \end{equation}
   where $p:=n-\ell$.
   Since $(1+\varepsilon)^p\ge(1+\varepsilon)^j$ for $j\le p$ and $c_k\ge0$, the above inequality implies that 
   \begin{equation}\label{eq:appPineq'}
       P^\ell_\omega(\delb\Psi_{\varepsilon,\delta})\le\left({1+\varepsilon}\right)^p\le d,
   \end{equation}
   if $\varepsilon$ is small enough.
   Then, the sequence $\varphi_a:=\Psi_{\varepsilon_a,\delta_a}$ decreases to $\varphi$ as $\varepsilon_a\searrow0$ and $\delta_a\searrow0$.

   The same arguments also yield \ref{item:appQ}. Indeed, we first take a finer cover of $\Omega'$ so that on each $B^i_{2r}$ we have a constant $c_{0,0}^i$ such that
   \begin{equation}\label{eq:c_00^i}
       -\varepsilon_{2.1}<c_{0,0}^i\le c_0\le (1+\varepsilon')c_{0,0}^i
   \end{equation}
   in addition to \eqref{eq:constomega}, where $\varepsilon'$ will be chosen later. Denote the function \eqref{eq:psidelta} by $\psi_{\varepsilon,\delta}^i$, which is in $\bar{\Gamma}'_{\omega_0}$ by assumption and Proposition \ref{prop:equivalence}. Using \eqref{eq:constomega} and \eqref{eq:c_00^i}, we see that
   \begin{equation}\label{eq:appQineq}
       (\delb\psi^i_{\varepsilon,\delta})^n-\sum_{k=1}^{n-1}\frac{c_k}{(1+\varepsilon)^{n-k}}(\delb\psi^i_{\varepsilon,\delta})^{k}\wedge\omega^{n-k}-\frac{c_0}{1+\varepsilon'}(\omega_0^i)^n\ge0.
   \end{equation}
   To handle the last term, we claim that
   \begin{equation}\label{eq:masspositivity}
        (\delb \psi^i_{\varepsilon,\delta})^n\ge c\omega^n
    \end{equation}
    for some constant $c>0$, independent of $i,\varepsilon,\delta$. Indeed, if $c_k=0$ for all $k=1,\dots,n-1$, then inequalities $Q_{c_{0,0}^i,\omega_0}(\delb\psi_{\varepsilon,\delta}^i)\le 1$, \eqref{eq:MA}, \eqref{eq:constomega}, and \eqref{eq:c_00^i} imply the claim. If 
    $c_k\neq0$ for some $k$, then since $P_{\omega_0}(\delb \psi^i_{\varepsilon,\delta})\le 1$, we have
    \begin{equation*}
         S_{n-1;i}(\lambda)\ge\sum_{k=1}^{n-1}c_k \frac{S_{k-1;i}(\lambda)}{{n\choose k}}\ge\sum_{k=1}^{n-1}c_k \frac{(S_{n-1;i}(\lambda))^{\frac{k-1}{n-1}}}{{n\choose k}},
    \end{equation*}
    where $\lambda$ is the eigenvalue vector of $\delb \psi^i_{\varepsilon,\delta}\cdot\omega_0^{-1}$. Here, the second inequality follows from the Newton-Maclaurin inequality \cite[Lemma 2.10]{Spr}. 
    In particular, we get $S_{n-1;i}(\lambda)>c$ for some constant $c>0$. Using the Newton-Maclaurin inequality again, we have $S_1(\lambda)\ge S_{1;i}(\lambda)>c^{1/{((n-1)}}$. Therefore, we have $S_n(\lambda)=\sum_i\lambda_i S_{n-1;i}(\lambda)/n>c^{n/(n-1)}/n$, which implies \eqref{eq:masspositivity} by \eqref{eq:constomega}. 
    Fix a constant $d'$ such that $1<d'<d$.
   Then, by \eqref{eq:masspositivity}, we have an estimate
    $$\left\vert \frac{c_0}{(1+\varepsilon')}(\omega_0^i)^n-c_0\omega^n\right\vert\le\vert c_0\vert\vert(\omega^i_0)^n-\omega^n\vert+\left(1-\frac{1}{1+\varepsilon'}\right)\omega^n\le (d'-1)(\delb\psi^i_{\varepsilon,\delta})^n$$
   if $\varepsilon$ and $\varepsilon'$ is small enough. 
   Applying this inequality to \eqref{eq:appQineq} and choosing $\varepsilon$ small enough to satisfy $d'(1+\varepsilon)^{n-1}<d$, by the arguments of the regularized maximum in \cite[p561--562]{GChen} with the convexity of the sublevel sets and the monotonicity of $Q_{c_0,\omega}$, we obtain the desired result.
\end{proof}

Using the approximation \ref{item:appP}, we obtain an inequality analogous to the comparison principle for the Monge-Ampère product \cite[Theorem 4.1]{BT}.

\begin{lem}\label{lem:comp}
    Let $\Omega$ and $\Omega'$ be bounded domains in $\mathbb{C}^n$ such that $\Omega'\subset\subset\Omega$, and $\omega$ be a Kähler form in $\Omega$. Suppose that $u_0,u_1$ are bounded psh functions on $\Omega$ such that $P^{\ell}_\omega(\delb u_0)\le_v1$ and $P^{\ell}_\omega(\delb u_1)\le_v1$ on $\Omega$ and $\liminf_{z\to\partial\Omega'} (u_1-u_0)\ge 0$. Then, we have
    \begin{equation}
        \int_{\{u_1<u_0\}\cap\Omega'}\omega^{n-p}\wedge T^p_\omega(\delb u_0)\le\int_{\{u_1<u_0\}\cap\Omega'}\omega^{n-p}\wedge T^p_\omega(\delb u_1),
    \end{equation}
    where $p:=n-\ell+1$ and $T^P_\omega$ is the $(p,p)$-current defined in Definition \ref{def:llcurrent}.
\end{lem}

\begin{proof}
    We follow the arguments in \cite[Theorem 4.1]{BT} with some modifications.
    Let $u_s:=(1-s)u_0+su_1$. Then, since
    \begin{equation*}
        \begin{aligned}
            &\frac{d}{ds}\int_{\{u_1<u_0\}}\omega^{n-p}\wedge T^p_\omega(\delb u_s)\\
            =&\int_{\{u_1<u_0\}}\omega^{n-p}\wedge\big(\delb (u_1-u_0)\big)\wedge T^{p-1}_\omega(\delb u_s),
        \end{aligned}
    \end{equation*}
    it suffices to prove
    \begin{equation}\label{eq:lincomp'}
        \begin{aligned}
            &\int_{\{u_1<u_0\}}\omega^{n-p}\wedge\delb u_0\wedge T^{p-1}_\omega(\delb u_s)\\
            \le&\int_{\{u_1<u_0\}}\omega^{n-p}\wedge\delb u_1\wedge T^{p-1}_\omega(\delb u_s).
        \end{aligned}
    \end{equation}
    
    By Lemma \ref{lem:app}, for any $d>1$, we have sequences of smooth psh functions $\{u_{0,a}\}$ and $\{u_{1,b}\}$ decreasing to $u_0$ and $u_1$, respectively, such that $P^{\ell}_\omega(\delb u_{0,a})\le d$ and $P^{\ell}_\omega(\delb u_{1,b})\le d$. Moreover, we can assume that $\liminf_{z\to\partial\Omega'}(u_1-u_0)\ge 2\delta>0$. Otherwise, replace $u_1$ by $u_1+2\delta$ and then let $\delta\to0$. Thus, there is an open set $\Omega''\subset\subset\Omega'$ such that $u_1(z)>u_0(z)+\delta$ for $z\in\Omega'\setminus\Omega''$. Hence, we can assume $u_{1,b}>u_{0,a}$ on $\partial\Omega''$. Now, we will focus on obtaining \eqref{eq:lincomp'} on $\Omega''$, since letting $\Omega''\to\Omega'$ then yields the same result in $\Omega$.
    Define $u_{s,a,b}:=(1-s)u_{0,a}+su_{1,b}$. 
    We first prove that 
    \begin{equation}\label{eq:lincomp}
        \begin{aligned}
            &\int_{\{u_{1,a}<u_{0,b}\}}\omega^{n-p}\wedge\delb u_{0,a}\wedge \left(T^{p-1}_{\omega}(\delb u_{s,a,b})+(d-1)(\delb u_{s,a,b})^{p-1}\right)\\
            \le&\int_{\{u_{1,a}<u_{0,b}\}}\omega^{n-p}\wedge\delb u_{1,b}\wedge \left(T^{p-1}_{\omega}(\delb u_{s,a,b})+(d-1)(\delb u_{s,a,b})^{p-1}\right).
        \end{aligned}
    \end{equation}
    Note that by the convexity of the sublevel set of $P^{\ell}$, we have $P^\ell_\omega(\delb u_{s,a,b})\le d$, which is equivalent to 
    $$T^{p-1}_\omega(\delb u_{s,a,b})+(d-1)(\delb u_{s,a,b})^{p-1}\ge0.$$ 
    Then, the arguments in the proof of \cite[Proposition 3.1]{BT'} replacing $\theta$ there with the semipositive $(n-1,n-1)$-form
    $$\omega^{n-p}\wedge T^{p-1}_\omega(\delb u_{s,a,b})+(d-1)(\delb u_{s,a,b})^{p-1}$$
    yield \eqref{eq:lincomp}. Applying the arguments in the proof of \cite[Theorem 4.1]{BT} as $a,b\to\infty$ and then letting $d\to1$ yield the desired inequality \eqref{eq:lincomp'}. Here, when applying the arguments of \cite[Theorem 4.1]{BT}, remark that we have
    $$\int_E \bigwedge_{i=1}^3(\delb\psi_i)^{j_i}\le \int_E \Big(\delb \Big(\sum_{i=1}^3\psi_i\Big)\Big)^n\le C \operatorname{cap}(E)$$
    for any set of bounded psh functions $\{\psi_i\}_{i=1}^3$ and set of nonnegative integers $\{j_i\}_{i=1}^3$ that satisfies $\sum_ij_i=n$,
    where $C$ is a constant depending on $\operatorname{osc}\psi_i$.
\end{proof}

\begin{cor}\label{cor:comp}
    Let $\Omega$ and $\Omega'$ be bounded domains in $\mathbb{C}^n$ such that $\Omega'\subset\subset\Omega$, and $\omega$ be a Kähler form in $\Omega$. Suppose that $u_0,u_1$ are bounded psh functions in $\Omega$ such that 
    \begin{equation*}
        \begin{cases}
            P^{\ell}_\omega(\delb u_i)\le_v1 \ \text{ in }\Omega \ \text{ for } i=1,2,\\
            T^p_\omega(\delb u_1)\le T^p_\omega(\delb u_0) \ \text{ in }\Omega',\\
            \liminf_{z\to\partial\Omega'} (u_1-u_0)\ge 0,
        \end{cases}
    \end{equation*}
    where $p=n-\ell+1$. Then, we have $u_0\le u_1$ in $\Omega'$.
\end{cor}

\begin{remark}
    The viscosity version of the corollary is proved in \cite[Theorem 22]{DDT}.
\end{remark}

\begin{proof}[Proof of Corollary \ref{cor:comp}]
    We follow the arguments in \cite[Corollary 4.4]{BT} with some modifications. Let $\psi\le0$ be a bounded strictly psh function in $\Omega$ and $S:=\{u_1<u_0+\varepsilon\psi\}\cap\Omega'$, where $\varepsilon>0$ is chosen later. By Lemma \ref{lem:comp}, we have
    \begin{equation*}
        \begin{aligned}
            &\int_S \omega^{n-p}\wedge T^p_\omega(\delb (u_0+\varepsilon\psi))\le\int_S \omega^{n-p}\wedge T^p_\omega(\delb u_1)\\
            \le&\int_{S}\omega^{n-p}\wedge T^p_\omega(\delb u_0).      
        \end{aligned}
    \end{equation*}
    Moreover, by the definition of $T^p_\omega$ (Definition \ref{def:llcurrent}), the left-hand side of the first line can be calculated as
    \begin{align*}
        &\int_S \omega^{n-p}\wedge T^p_\omega(\delb (u_0+\varepsilon\psi))=\sum_{i=0}^p\frac{\varepsilon^i}{i!}\int_S \omega^{n-p}\wedge T^{p-i}_\omega(\delb u_0)\wedge(\delb\psi)^i\\
        \ge&\int_S \omega^{n-p}\wedge T^p_\omega(\delb u_0)+\frac{\varepsilon^p}{p!}\int_S\omega^{n-p}\wedge(\delb \psi)^p,
    \end{align*}
    where the inequality follows since $T^{p-i}_\omega(\delb u_0)$ is positive by Lemma \ref{lem:visctopp}. Since $\psi$ is strictly psh, we conclude that $S$ has a Lebesgue measure zero. Since $u_0,u_1$ are subharmonic, using the upper semicontinuity and sub-mean value inequality, we obtain $S=\emptyset$. Letting $\varepsilon\searrow0$ yields the desired claim.
\end{proof}

\begin{prop}\label{prop:viscpp}
    Let $\varphi$ be a $\chi$-psh function.
    \begin{enumerate}[label={\upshape(\roman*)}]
        \item\label{item:posin-1} 
        $P^\ell_{\omega}(\chi_\varphi)\le_v 1$ in $\{\omega>0\}$ if and only if $T^p_\omega(\chi_\varphi)$ is positive in $X$ for all $p\le n-\ell$.
        \item\label{item:posin} 
        Suppose $P_{\omega}(\chi_\varphi)\le_v 1$. Then, $Q_{c_0,\omega}(\chi_\varphi)\le_v 1$ in $\{\omega>0\}$ if and only if $T^n_\omega(\chi_\varphi)$ is positive in $X$.
    \end{enumerate}
\end{prop}

\begin{proof}
    Since we already proved the part ``only if'' in Lemma \ref{lem:visctopp}, it suffices to prove the converse implication.
    
    We prove \ref{item:posin-1}, following the arguments in the proofs of \cite[Propositions 1.5 and 1.11]{EGZ} with some modifications. The proof of \ref{item:posin} is the same, so we omit it here. Suppose that $T^p_\omega(\chi_\varphi)\ge0$ in $X$ for all $p\le n-\ell$. 
    We use induction in $\ell$ to prove $P^\ell_\omega(\chi_\varphi)\le_v1$ in $\{\omega>0\}$. In the case of $\ell=n-1$, the assumption $T^1_\omega(\chi_\varphi)\ge0$ is equivalent to saying that a local upper semicontinous potential of $\chi_\varphi-c_{n-1}\omega$ is psh. This property implies $\chi_\varphi-c_{n-1}\omega\ge0$ in the viscosity sense (see, e.g. the proof of \cite[Proposition 1.3]{EGZ}), which is the desired claim.
    Now, assume $P^{\ell-1}_\omega(\chi_\varphi)\le_v1$ and suppose that $P^\ell_\omega(\chi_\varphi)\le_v1$ does not hold. Then, there exist a point $p\in\{\omega>0\}$ and an upper test $q$ of an upper semicontinuous local potential $u$ of $\chi_\varphi$ at $p$ such that $P^\ell_\omega(\delb q)(p)>1$ in $U$, a neighborhood of $p$. In particular, we have $u(p)=q(p)>-\infty$.
    Then, we can choose $\varepsilon>0$ and $r>0$ so small that 
    $$q_\varepsilon:=q+\varepsilon\vert z\vert^2-\frac{\varepsilon r^2}{2}$$
    satisfies $P^\ell_\omega(\delb q_\varepsilon)>1$ in $B:=B_{r}(p)$, a ball centered at $p$ with radius $r$. Then, we choose $\varepsilon'$ such that $u(p)+\varepsilon'v(p)>q_\varepsilon(p)$, where $v$ is a nonpositive local upper semicontinuous potential of $\omega$. 
    Furthermore, we choose $C$ large enough to satisfy
    \begin{align}
        &\label{eq:inB} q_\varepsilon\ge av-C+\varepsilon'v \text{ in } B,\\
        &\label{eq:atp} u(p)>av(p)-C,\\
        &\label{eq:bdry}q_\varepsilon> av-C \text{ on }\partial B,
    \end{align}
    where $a$ is a constant that satisfies $P^{\ell-1}_\omega(a\omega)\le1$.
    Then, by \eqref{eq:atp}, we have $p\in S:=\{q_\varepsilon<u^C+\varepsilon'v\}$, where $u^C:=\sup\{u,av-C\}$. By \eqref{eq:inB}, we have $S\subset\{u>av-C\}$. In particular, we have $$T^p_\omega(\delb u^C)=T^p_\omega(\delb u)\ge T^p_\omega(\delb q_\varepsilon) \text{ in } S.$$ By \eqref{eq:bdry} and $q_\varepsilon-u\ge \varepsilon r^2/4$ near $\partial B$, we have $\liminf_{z\to\partial B}(q_\varepsilon-u^C-\varepsilon'v)\ge0$. Therefore, we see that 
    \begin{equation*}
        \begin{aligned}
            &\int_S \omega^{n-p}\wedge T^p_\omega(\delb (u^C+\varepsilon' v))\le\int_S \omega^{n-p}\wedge T^p_\omega(\delb q_\varepsilon)\\
            \le&\int_{S}\omega^{n-p}\wedge T^p_\omega(\delb u^C),      
        \end{aligned}
    \end{equation*}
    where the first inequality follows from Lemma \ref{lem:comp}, since $P^{\ell-1}_\omega(\delb q_\varepsilon)\le_v1$ and $P^{\ell-1}_\omega(\delb u^C)\le_v1$. By the same argument as in the proof of Corollary \ref{cor:comp}, we conclude that $S=\emptyset$, which is a contradiction to $p\in S$. 
\end{proof}

Now, we are on the path to prove the uniqueness. The proof is motivated by that of the MA equation as in \cite[Section 3]{BEGZ} and \cite{Dinew}. To follow their arguments, we need two lemmas.

\begin{lem}\label{lem:uni}
     Suppose that $\varphi_1$ and $\varphi_2$ are solutions of \eqref{eq:gMA}. Then, for $i=1,2$, we have
     $$\left\langle\chi_{\varphi_1}\wedge T^{n-1}_\omega(\chi_{\varphi_i})\right\rangle=\left\langle\chi_{\varphi_2}\wedge T^{n-1}_\omega(\chi_{\varphi_i})\right\rangle.$$
\end{lem}

\begin{proof}
    By Proposition \ref{prop:viscpp} and Proposition \ref{prop:equivalence}, we have $P_\omega(\chi_{\varphi_i})\le_v1$ and $Q_{c_0,\omega}(\chi_{\varphi_i})\le_v 1$ for $i=1,2$. By the convexity of the sublevel sets of $P_\omega$ and $Q_{c_0,\omega}$, we have $P_\omega((1-s)\chi_{\varphi_1}+s\chi_{\varphi_2})\le_v 1$ and $Q_{c_0,\omega}((1-s)\chi_{\varphi_1}+s\chi_{\varphi_2})\le_v 1$ for any $s\in[0,1]$ (see, e.g., the proof of \cite[Theorem 5.8]{CC}). Then, interpreting this property in terms of the nonpluripolar product by Lemma \ref{lem:visctopp} and using Proposition \ref{prop:superineq}, we see that 
    $T^n_\omega((1-s)\chi_{\varphi_1}+s\chi_{\varphi_2})=0$. By differentiating this equation in $s$ at $s=0$ and $s=1$, we obtain the desired result.
\end{proof}

\begin{lem}\label{lem:solineq}
    Let $\chi_\infty:=\chi+\delb\varphi_\infty$ be a solution of \eqref{eq:gMA}. Then, for any $\chi$-psh function $\varphi$ such that $P_\omega(\chi_\varphi)\le_v 1$, we have
    $$\int_X\big\langle\chi_\varphi\wedge T^{n-1}_\omega(\chi_\infty)\big\rangle\le \int_X\big\langle \chi_\infty\wedge T^{n-1}_\omega(\chi_\infty)\big\rangle.$$
\end{lem}

\begin{proof}
    Define $\varphi_s:=(1-s)\varphi_\infty+s\varphi$. Note that
    \begin{equation*}
        \frac{d}{ds}\Bigg\vert_{s=0}\int_X T^n_\omega(\chi_{\varphi_s})=\int_X \big\langle(\chi_\varphi-\chi_\infty)\wedge T^{n-1}_\omega(\chi_{\varphi_\infty})\big\rangle.
    \end{equation*}
    Since $T^n_\omega(\chi_\infty)=0$ and $P_{\omega}(\chi_{\varphi_s})\le_v1$, by Proposition \ref{prop:superineq}, we have that the left-hand side is nonpositive, which concludes the desired result.
\end{proof}

\begin{proof}[Proof of the uniqueness part of Theorem \ref{thm:main}]
    We follow the arguments by Dinew \cite{Dinew} (see also \cite[Section 3.3]{BEGZ}). Suppose that $\varphi_1$ and $\varphi_2$ are solutions of the gMA equation \eqref{eq:gMA}. We fix $i=1$ or 2.
    
    First, suppose that there exists a constant $a$ such that $\varphi_1=\varphi_2+a$ almost everywhere with respect to the measure $\langle\chi_{\varphi_i}\wedge T^{n-1}_\omega(\chi_{\varphi_i})\rangle$. Note that we have 
    \begin{equation}\label{eq:masspositivity'}
        \langle\chi_{\varphi_i}^n\rangle\ge c\omega^n
    \end{equation}
    for some constant $c>0$. Indeed, if $c_k=0$ for all $k=1,\dots,n-1$, then $\langle\chi_{\varphi_i}^n\rangle=c_0\omega^n$, and \eqref{eq:MA} implies \eqref{eq:masspositivity'}. If $c_k\neq0$ for some $k$, by the arguments of \eqref{eq:masspositivity}, we have $S_n(\lambda)>c^{n/(n-1)}/n$
    for any upper test $q$ at $p\in\{\omega>0\}$, where $\lambda$ is the eigenvalue vector of $\delb q$ at $p$. 
    By the same argument of Lemma \ref{lem:visctopp} or \cite[Corollary 2.6]{EGZ}, this implies \eqref{eq:masspositivity'}.
    Then, since
    \begin{equation*}
        \begin{aligned}
            \langle\chi_{\varphi_i}\wedge T^{n-1}_\omega(\chi_{\varphi_i})\rangle
            =n\langle\chi_{\varphi_i}^n\rangle-\sum_{k=1}^{n-1}c_k k\langle\chi_{\varphi_i}^k\wedge\omega^{n-k}\rangle\ge\langle\chi_{\varphi_i}^n\rangle,
        \end{aligned}
    \end{equation*}
    we conclude that $\varphi_1=\varphi_2+a$ almost everywhere with respect to Lebesgue measure, which is the desired result.
    
    Next, suppose that there is no constant $a$ such that $\varphi_1=\varphi_2+a$ almost everywhere with respect to the measure $\langle\chi_{\varphi_i}\wedge T^{n-1}_\omega(\chi_{\varphi_i})\rangle$. Then, as in the proof of \cite[Corollary 1.10]{GZ}, since the function
    $$f:\mathbb{R}\ni t\mapsto \langle\chi_{\varphi_i}\wedge T^{n-1}_\omega(\chi_{\varphi_i})\rangle(\{\varphi_1<\varphi_2+t\})\in\mathbb{R}^+$$
    is left continuous, monotone increasing, we can find $a\in\mathbb{R}$ such that  
    \begin{equation}\label{eq:mass}
    \begin{split}
        &0<\left\langle\chi_{\varphi_i}\wedge T^{n-1}_\omega(\chi_{\varphi_i})\right\rangle(\{\varphi_1<\varphi_2+a\})<v:=\int_X\left\langle\chi_{\varphi_i}\wedge T^{n-1}_\omega(\chi_{\varphi_i})\right\rangle,\\
        &\left\langle\chi_{\varphi_i}\wedge T^{n-1}_\omega(\chi_{\varphi_i})\right\rangle(\{\varphi_1=\varphi_2+a\})=0.
    \end{split}
    \end{equation}
    Without loss of generality, we replace $\varphi_2$ by $\varphi_2+a$.
    Since if $\varphi_1=\varphi_2$ almost everywhere with respect to $\omega^n$ we are done, by changing $a$ if necessary, we can further assume that
    \begin{equation}\label{eq:lebesguemass}
        \begin{split}
        &0<\omega^n(\{\varphi_1<\varphi_2\})<v':=\int_X\omega^n,\\
        &\omega
        ^n(\{\varphi_1=\varphi_2\})=0.
    \end{split}
    \end{equation}
    We first claim that there exists a bounded lower semicontinuous function $f_1$ such that 
    \begin{equation}\label{eq:f_1}
        \begin{cases}
            f_1>\varepsilon>0 \text{ on } U\supset \{\varphi_1<\varphi_2\}, \\
            c_0+f_1>-\varepsilon_{2.1} \text{ in } X,  \\
            \int_X f_1 \omega^n=0,
        \end{cases}
    \end{equation}
    where $U$ is an open subset containing $\{\varphi_1<\varphi_2\}$ such that $\omega^n(U)<v'$ and $\varepsilon$ is a small constant chosen below. Indeed, the existence of such $U$ can be seen as follows. Cover $X$ by a finite number of coordinate balls $\{B^i\}^N_{i=1}$. In each ball $B^i$, by quasicontinuity of psh functions \cite[Theorem 3.5]{BT}, there exists an open subset $E^i$ with $\operatorname{cap}(E^i)<\delta$ such that $\varphi_1$ and $\varphi_2$ are continuous in $B^i\setminus E^i$. Note that by the definition of the capacity, we have $\omega^n(E^i)\le C \operatorname{cap}(E^i)<C\delta$. Then, we have $\{\varphi_1<\varphi_2\}\cap (B^i\setminus E^i)= U^i\setminus E^i$ for some open subset $U^i\subset B^i$. Running this argument for all $i$, we see that $\omega^n(\cup_i (U^i\setminus E^i))\le\omega^n(\{\varphi_1<\varphi_2\})<v'$, where the last inequality is \eqref{eq:lebesguemass}. Hence, by defining $U:=\cup_i (U^i\cup E^i)$, we have $U\supset\{\varphi_1<\varphi_2\}$ and 
    $$\omega^n(U)\le\omega^n(\cup_i (U^i\setminus E^i))+\sum_{i=1}^N\omega^n(E^i)<\omega^n(\{\varphi_1<\varphi_2\})+CN\delta<v',$$
    where the last inequality follows by taking $\delta>0$ small enough. Due to $\omega^n(U)<v'$, if $\varepsilon$ is small enough, we can define $f_1$ as $f_1:=\varepsilon'>\varepsilon$ on $U$ and $f_1:=-\varepsilon''<0$ on $X\setminus U$, which satisfies \eqref{eq:f_1}.
    By the existence of a solution of the gMA equation with a bounded lower semicontinuous function (Theorem \ref{thm:existence}), there exists a $\chi$-psh function $\varphi'_1$ such that
    $$\langle\chi_{\varphi'_1}\rangle^n=\sum_{k=0}^{n-1}c_k\langle\chi_{\varphi'_1}^k\wedge\omega^{n-k}\rangle+f_1\omega^n.$$
    By switching $\varphi_1$ and $\varphi_2$, the same arguments provide a smooth function $f_2$ and a $\chi$-psh function $\varphi'_2$ with the same properties. Let $\varphi_{1,2}:=\sup\{\varphi_1,\varphi_2\}$ and $\varphi'_{1,2}:=\sup\{\varphi_1',\varphi_2'\}$. Note that we know $\varphi_{1,2}$ is also a solution of the gMA equation, by Propositions \ref{prop:equivalence}, \ref{prop:viscpp} and \ref{prop:superineq}.
    We first claim that
    \begin{equation}\label{ineqJcomp}
    \begin{split}
        &\int_{\{(1-\varepsilon)\varphi_1+\varepsilon\varphi_{1,2}'<(1-\varepsilon)\varphi_2+\varepsilon\varphi_1'\}}
        \left\langle\chi_{(1-\varepsilon)\varphi_2+\varepsilon\varphi_1'}\wedge T^{n-1}_\omega(\chi_{\varphi_i})\right\rangle\\
        \le &\int_{\{(1-\varepsilon)\varphi_1+\varepsilon\varphi_{1,2}'<(1-\varepsilon)\varphi_2+\varepsilon\varphi_1'\}}
        \left\langle\chi_{(1-\varepsilon)\varphi_1+\varepsilon\varphi_{1,2}'}\wedge T^{n-1}_\omega(\chi_{\varphi_i})\right\rangle.
    \end{split}
    \end{equation}
    Indeed, for any $\delta>0$, we have
    \begin{equation*}
        \begin{split}
            &\int_X \left\langle\chi_{\sup\{(1-\varepsilon)\varphi_2+\varepsilon\varphi'_{1,2},(1-\varepsilon)\varphi_1+\varepsilon\varphi'_1\}}\wedge T^{n-1}_\omega(\chi_{\varphi_i})\right\rangle\\
            \ge&\int_{\{(1-\varepsilon)\varphi_1+\varepsilon\varphi_{1,2}'<(1-\varepsilon)\varphi_2+\varepsilon\varphi_1'-\delta\}}\left\langle\chi_{(1-\varepsilon)\varphi_2+\varepsilon\varphi_1'}\wedge T^{n-1}_\omega(\chi_{\varphi_i})\right\rangle\\
            &+\int_{\{(1-\varepsilon)\varphi_1+\varepsilon\varphi_{1,2}'>(1-\varepsilon)\varphi_2+\varepsilon\varphi_1'-\delta\}}\left\langle\chi_{(1-\varepsilon)\varphi_1+\varepsilon\varphi_{1,2}'}\wedge T^{n-1}_\omega(\chi_{\varphi_i})\right\rangle\\
            =&\int_{\{(1-\varepsilon)\varphi_1+\varepsilon\varphi_{1,2}'<(1-\varepsilon)\varphi_2+\varepsilon\varphi_1'-\delta\}}\left\langle\chi_{(1-\varepsilon)\varphi_2+\varepsilon\varphi_1'}\wedge T^{n-1}_\omega(\chi_{\varphi_i})\right\rangle\\
            &+\int_{X\setminus{\{(1-\varepsilon)\varphi_1+\varepsilon\varphi_{1,2}'\le(1-\varepsilon)\varphi_2+\varepsilon\varphi_1'-\delta\}}}\left\langle\chi_{(1-\varepsilon)\varphi_1+\varepsilon\varphi_{1,2}'}\wedge T^{n-1}_\omega(\chi_{\varphi_i})\right\rangle\\
            \ge&\int_{\{(1-\varepsilon)\varphi_1+\varepsilon\varphi_{1,2}'<(1-\varepsilon)\varphi_2+\varepsilon\varphi_1'-\delta\}}\left\langle\chi_{(1-\varepsilon)\varphi_2+\varepsilon\varphi_1'}\wedge T^{n-1}_\omega(\chi_{\varphi_i})\right\rangle\\
            &+\int_X\left\langle\chi_{(1-\varepsilon)\varphi_1+\varepsilon\varphi_{1,2}'}\wedge T^{n-1}_\omega(\chi_{\varphi_i})\right\rangle\\
            &-\int_{\{(1-\varepsilon)\varphi_1+\varepsilon\varphi_{1,2}'<(1-\varepsilon)\varphi_2+\varepsilon\varphi_1'\}}\left\langle\chi_{(1-\varepsilon)\varphi_1+\varepsilon\varphi_{1,2}'}\wedge T^{n-1}_\omega(\chi_{\varphi_i})\right\rangle
        \end{split}
    \end{equation*}
    By letting $\delta$ approach zero, we get
    \begin{equation*}
        \begin{split}
            &\int_{\{(1-\varepsilon)\varphi_1+\varepsilon\varphi_{1,2}'<(1-\varepsilon)\varphi_2+\varepsilon\varphi_1'\}}\left\langle\chi_{(1-\varepsilon)\varphi_2+\varepsilon\varphi_1'}\wedge T^{n-1}_\omega(\chi_{\varphi_i})\right\rangle\\
        \le &\int_{\{(1-\varepsilon)\varphi_1+\varepsilon\varphi_{1,2}'<(1-\varepsilon)\varphi_2+\varepsilon\varphi_1'\}}\left\langle\chi_{(1-\varepsilon)\varphi_1+\varepsilon\varphi_{1,2}'}\wedge T^{n-1}_\omega(\chi_{\varphi_i})\right\rangle\\
        +&\int_X \left\langle\chi_{\sup\{(1-\varepsilon)\varphi_2+\varepsilon\varphi'_{1,2},(1-\varepsilon)\varphi_1+\varepsilon\varphi'_1\}}\wedge T^{n-1}_\omega(\chi_{\varphi_i})\right\rangle
        -\int_X\left\langle\chi_{(1-\varepsilon)\varphi_1+\varepsilon\varphi_{1,2}'}\wedge T^{n-1}_\omega(\chi_{\varphi_i})\right\rangle
        \end{split}
    \end{equation*}
    On the other hand, we have
    \begin{equation*}
        \begin{split}
            &\int_X \left\langle\chi_{\sup\{(1-\varepsilon)\varphi_2+\varepsilon\varphi'_{1,2},(1-\varepsilon)\varphi_1+\varepsilon\varphi'_1\}}\wedge T^{n-1}_\omega(\chi_{\varphi_i})\right\rangle\\
            \le&\int_X \left\langle\chi_{(1-\varepsilon)\varphi_{1,2}+\varepsilon\varphi'_{1,2}}\wedge T^{n-1}_\omega(\chi_{\varphi_i})\right\rangle
            =\int_X \left\langle\chi_{(1-\varepsilon)\varphi_2+\varepsilon\varphi'_{1,2}}\wedge T^{n-1}_\omega(\chi_{\varphi_i})\right\rangle.
        \end{split}
    \end{equation*}
    Here, the inequality follows from Lemma \ref{lem:solineq} and the equality follows from Lemma \ref{lem:uni} and the fact $\varphi_{1,2}$ is a solution of the gMA equation \eqref{eq:gMA}. Thus, we have proved \eqref{ineqJcomp}. Since $\varphi_1,\varphi_2$ are solutions of the gMA equation, by Lemma \ref{lem:uni}, we have
    $$\left\langle\chi_{\varphi_1}\wedge T^{n-1}_\omega(\chi_{\varphi_i})\right\rangle=\left\langle\chi_{\varphi_2}\wedge  T^{n-1}_\omega(\chi_{\varphi_i})\right\rangle.$$
    By applying this to \eqref{ineqJcomp} and letting $\varepsilon$ approach zero, we get
    \begin{equation*}
        \begin{split}
            \int_{\{\varphi_1<\varphi_2\}}\left\langle\chi_{\varphi_1'}\wedge T^{n-1}_\omega(\chi_{\varphi_i})\right\rangle
            \le\int_{\{\varphi_1<\varphi_2\}}\left\langle\chi_{\varphi_{1,2}'}\wedge T^{n-1}_\omega(\chi_{\varphi_i})\right\rangle.
        \end{split}
    \end{equation*}
    Combining this inequality with Lemma \ref{lem:BMineq} below, we obtain
    \begin{equation*}
        \begin{aligned}
            b+\int_{\{\varphi_1<\varphi_2\}}\big\langle\chi_{\varphi_i}\wedge T^{n-1}_\omega(\chi_{\varphi_i})\big\rangle
            \le&\int_{\{\varphi_1<\varphi_2\}}\left\langle\chi_{\varphi_1'}\wedge T^{n-1}_\omega(\chi_{\varphi_i})\right\rangle
            \\
            \le&\int_{\{\varphi_1<\varphi_2\}}\left\langle\chi_{\varphi_{1,2}'}\wedge T^{n-1}_\omega(\chi_{\varphi_i})\right\rangle
        \end{aligned}
    \end{equation*}
    for some $b>0$.
    
    By switching $\varphi_1$ and $\varphi_2$, the same argument yields
    $$b+\int_{\{\varphi_2<\varphi_1\}}\big\langle\chi_{\varphi_i}\wedge T^{n-1}_\omega(\chi_{\varphi_i})\big\rangle\le \int_{\{\varphi_2<\varphi_1\}}\left\langle\chi_{\varphi_{1,2}'}\wedge T^{n-1}_\omega(\chi_{\varphi_i})\right\rangle.$$
    By adding these inequalities together, we obtain
    $$2b+\int_X\big\langle\chi_{\varphi_i}\wedge T^{n-1}_\omega(\chi_{\varphi_i})\big\rangle\le \int_X\left\langle\chi_{\varphi_{1,2}'}\wedge T^{n-1}_\omega(\chi_{\varphi_i})\right\rangle\le\int_X\big\langle\chi_{\varphi_i}\wedge T^{n-1}_\omega(\chi_{\varphi_i})\big\rangle,$$
    where the last inequality follows from Lemma \ref{lem:solineq}. This is a contradiction.
\end{proof}

\begin{lem}\label{lem:BMineq}
    We have
    $$b+\int_{\{\varphi_1<\varphi_2\}}\big\langle\chi_{\varphi_i}\wedge T^{n-1}_\omega(\chi_{\varphi_i})\big\rangle\le \int_{\{\varphi_1<\varphi_2\}}\left\langle\chi_{\varphi_1'}\wedge T^{n-1}_\omega(\chi_{\varphi_i})\right\rangle$$
    for some $b>0$.
\end{lem}

\begin{proof}
    Let F be a closed subset defined by $F:=\cup_i\{\varphi_1\le\varphi_2-c\}\cap (\overline{B'^i}\setminus E^i)$, where $E^i$ is the set that appeared in the arguments below \eqref{eq:f_1} and $\{B'^i\}$ is a finite number of coordinate balls that are slightly smaller than $B^i$ appearing in the arguments below \eqref{eq:f_1} but still cover $X$. We have $F\subset \{\varphi_1<\varphi_2\}$. Also, since $\omega^n(\{\varphi_1\le\varphi_2-c\})\to\omega^n(\{\varphi_1<\varphi_2\})$ as $c\to0$, using $\omega^n(E^i)<\delta$ as before, we have $\omega^n(F)>0$ for some $c>0$.
    Take a sequence of open subsets $U_j$ such that $U_j\subset\subset U_{j+1}$ and $U\cap\{\omega>0\}=\cup_j U_j$. Cover $U_j$ by a finite number of coordinate balls $\{B^{j,\ell}_r\}_\ell$ such that $B^{j,\ell}_{3r}\subset U_{j+1}$. Since we have $F\subset \{\varphi_1<\varphi_2\}\subset U$ and $\omega^n(F)>0$, there exists a ball $B^{j,\ell_0}_{r}$ such that $\omega^n(F\cap \bar{B}^{j,\ell_0}_r)>0$. In the following, we simply denote such a ball $B^{j,\ell_0}_r$ by $B_r$.
    Denote by $u$ and $u'$ upper semicontinuous potentials of $\chi_{\varphi_i}$ and $\chi_{\varphi_1'}$ on $B_{3r}$. Define $u^C:=\sup\{u,av-C\}$ and $u'^C:=\sup\{u',av'-C\}$, and denote by $u^C_a, u'^C_b$ the smooth approximations of $u^C$ with respect to $Q_{c_0,\omega}$ and $u'^C$ with respect to $Q_{c_0+\varepsilon,\omega}$ on $B_{2r}$ obtained by Lemma \ref{lem:app}. 
    We assume that a constant $C$ is so large that a closed subset $F'$, defined by $$F':=F\cap \bar{B}_r\cap\{u\ge av-C+1\}\cap\{u'\ge av'-C+1\}\setminus E,$$ satisfies $\omega^n(F')>0$, where $E$ is an open subset with a small capacity outside which $u$ and $u'$ are continuous. By \cite[Theorem 6.10]{BT}, we can find such a constant $C$.
    Define $u^C_{a,b,s}:=(1-s)u^C_a+s u'^C_b$. By the convexity of $Q_{c_0,\omega}$, for any $a$ and $b$, we have
    \begin{align*}
        \frac{d}{ds}\Bigg\vert_{s=0} Q_{c_0,\omega}(\delb u_{a,b,s}^C)\le Q_{c_0,\omega}(\delb u'^C_b)-Q_{c_0,\omega}(\delb u^C_a) \quad \text{in}\  B_{2r}.
    \end{align*}
    Since we have
    \begin{align*}
        Q_{c_0,\omega}(\delb u'^C_b)+f_1\frac{\omega^n}{(\delb u'^C_b)^n}=Q_{c_0+f_1,\omega}(\delb u'^C_b)\le d,
    \end{align*}
    The first term of the right hand side can be estimated as
    \begin{equation*}
        Q_{c_0,\omega}(\delb u'^C_b)\le d-\frac{\varepsilon\omega^n}{\left(\delb u'^C_b\right)^n} \text{ in } B_{2r}.
    \end{equation*}
    Since we have
    $$\frac{d}{ds}\Bigg\vert_{s=0} Q_{c_0,\omega}(\delb u^C_{a,b,s})=-\frac{(\delb u'^C_b-\delb u^C_a)\wedge T^{n-1}_\omega(\delb u^C_a)}{\big(\delb u^C_a\big)^n},$$
    we obtain
    \begin{equation}\label{eq:convexmass}
        \begin{aligned}
            &\left(\delb u'^C_b-\delb u^C_a\right)\wedge T^{n-1}_\omega(\delb u^C_a)\\
            \ge &\left(Q_{c_0,\omega}(\delb u^C_a)-Q_{c_0,\omega}(\delb u'^C_b)\right)(\delb u^C_a)^n \\ 
            \ge&-T^n_\omega(\delb u^C_a)-(d-1)(\delb u^C_a)^n+\varepsilon\frac{\omega^n}{(\delb u'^C_b)^n}(\delb u^C_a)^n\quad \text{in $B_{2r}$}.
        \end{aligned}
    \end{equation}
    Note that since $P_\omega(\delb u^C_a)\le1$, we have $(\delb u^C_a)^n\ge c\omega^n$ by \eqref{eq:masspositivity'}. 
    Hence, we see that 
    \begin{equation}\label{eq:perturbmass}
        \begin{aligned}
            &\int_{F'}\varepsilon\frac{\omega^n}{(\delb u'^C_b)^n}(\delb u^C_a)^n
        \ge\varepsilon c \int_{F'}\frac{\omega^n}{(\delb u'^C_b)^n}\omega^n\\
        \ge&\varepsilon c \left(\int_{F'}\omega^n\right)^2\Bigg/\left(\int_{F'}(\delb u'^C_b)^n\right),
        \end{aligned}
    \end{equation}
    where the second inequality follows from the Cauchy-Schwarz inequality. Since the weak convergence has the upper semicontinuity for the mass of closed sets, by letting $a\to\infty$ and $b\to\infty$ in \eqref{eq:convexmass} and using \eqref{eq:perturbmass} and \cite[Theorem 2.1]{BT}, we get
    \begin{equation*}
        \begin{aligned}
            &\int_{F'}\left(\delb u'^C-\delb u^C\right)\wedge T^{n-1}_\omega(\delb u^C)+T^n_\omega(\delb u^C)+(d-1)(\delb u^C)^n\\
            \ge&\liminf_{a,b\to\infty}\varepsilon c \left(\int_{F'}\omega^n\right)^2\Bigg/\left(\int_{F'}(\delb u'^C_b)^n\right)\ge\varepsilon c \left(\int_{F'}\omega^n\right)^2\Bigg/\left(\int_{F'}(\delb u'^C)^n\right)
        \end{aligned}
    \end{equation*}
    Since $F'\subset\{u\ge av-C+1\}\subset \{u>av-C\}$ and the same for $u'$, by the nonpluripolar locality, the above inequality can be rewritten as 
    \begin{equation}\label{eq:desiredineq}
        \begin{aligned}
            &\int_{F'}\Big\langle\left(\chi_{\varphi'_1}-\chi_{\varphi_i}\right)\wedge T^{n-1}_\omega(\chi_{\varphi_i})\Big\rangle+(d-1)\langle\chi_{\varphi_i}\rangle^n
            \ge\varepsilon c \left(\int_{F'}\omega^n\right)^2\Bigg/\left(\int_{F'}\big\langle\chi_{\varphi'_1}^n\big\rangle\right)\\
            \ge&\varepsilon c \left(\int_{F'}\omega^n\right)^2\Bigg/\left(\int_X\big\langle\chi_{\varphi'_1}^n\big\rangle\right)\ge\varepsilon c \left(\omega^n(F')\right)^2/\operatorname{vol}([\chi]^n)>0,
        \end{aligned}
    \end{equation}
    where the second to last inequality follows from \cite[Proposition 1.20]{BEGZ}. On the other hand, by the same arguments as above on any coordinate balls with estimating the last term in the last line in \eqref{eq:convexmass} by zero from below and letting $a,b\to\infty$ and then $C\to\infty$ on $\{u>av-C\}\cap\{u'>av'-C\}\cap B_{2r}$, we see that
    \begin{equation*}
        \Big\langle(\chi_{\varphi_1'}-\chi_{\varphi_i})\wedge T^{n-1}_{\omega}(\chi_{\varphi_i})\Big\rangle\ge -(d-1)\langle\chi_{\varphi_i}^n\rangle \ \text{ in } U\cap\{\omega>0\}.
    \end{equation*}
    Integrating this inequality in $(\{\varphi_1<\varphi_2\}\setminus F')\cap\{\omega>0\}$ and adding \eqref{eq:desiredineq}, we obtain
    \begin{equation*}
        \int_{\{\varphi_1<\varphi_2\}}\Big\langle\chi_{\varphi_1'}\wedge T^{n-1}_\omega(\chi_{\varphi_i})\Big\rangle\ge\int_{\{\varphi_1<\varphi_2\}}\Big\langle\chi_{\varphi
        _i}\wedge T^{n-1}_\omega(\chi_{\varphi_i})\Big\rangle-(d-1)\langle\chi_{\varphi_i}^n\rangle+b,
    \end{equation*}
    where $b:=\varepsilon c (\omega^n(F'))^2/\operatorname{vol}([\chi^n])$. Taking $d\to1$, we conclude the desired inequality.
\end{proof}

\subsection{Weak convergence of the mixed Hessian flow}\label{subsec:mixedHessianflow}
In this subsection, we prove Theorem \ref{thm:Hessflow}. As explained in Section \ref{sec:intro}, we follow the same lines as \cite{Mur}. Recall that the mixed Hessian flow is given by
\begin{equation}\label{eq:mixedHessianflowre}
    \begin{cases}
        \dot{\varphi_t}=1-\sum_{k=0}^{n-1}c_k\frac{\chi_t^k\wedge\omega^{n-k}}{\chi_t^n},\\
        \varphi_t|_{t=0}=\varphi_0.
    \end{cases}
\end{equation}

We start with the following lemma, which is a generalization of the case of the $J$-equation by \cite[Lemma 1]{Hashimoto}.

\begin{lem}\label{lem:F}
    Let $F_{k,\chi,\omega}$ be the operator defined by
    \begin{align*}
        F_{k,\chi,\omega}(\varphi)=&-\frac{\delb \varphi\wedge k\chi^{k-1}\wedge\omega^{n-k}}{\chi^n}+\frac{\chi^k\wedge\omega^{n-k}}{\chi^n}n\frac{\delb\varphi\wedge\chi^{n-1}}{\chi^n}\\
        &+\left(\partial\frac{\chi^k\wedge\omega^{n-k}}{\chi^n},\bar{\partial}\varphi\right)_{\chi}.
    \end{align*}
    Then, we have 
    $$\int_X\varphi F_{k,\chi,\omega}(\varphi)\ge0$$
    for any $\varphi\in C^\infty(X)$.
\end{lem}

\begin{proof}
For $\varphi,\psi\in C^\infty(X)$, by integration by parts, we can compute as
\begin{align*}
& \int_X {\varphi \, F_{k,\chi,\omega} (\psi)} \omega^{n} \\
=&-\int_X \varphi \left({\sqrt{-1} \partial \bar{\partial} \psi} \right)\wedge k\chi^{k-1} \wedge \omega^{n-k} 
+\int_X n \, \varphi \frac{\chi^k\wedge\omega^{n-k}}{\chi^n} \left( \sqrt{-1} \partial \bar{\partial} \psi \right) \wedge \chi^{n-1} \\
& \qquad + \int_X n \, \varphi \, \left( \sqrt{-1} \partial \frac{\chi^k\wedge\omega^{n-k}}{\chi^k} \wedge \bar{\partial}\psi \right) \wedge \chi^{n-1} \\
= & \int_X \left( \sqrt{-1} \partial \phi \wedge \bar{\partial} \psi \right) \wedge \left(k\chi^{k-1} \wedge \omega^{n-k} -\frac{\chi^k\wedge\omega^{n-k}}{\chi^n} \chi^{n-1}\right).
\end{align*}
Since $S_k(\lambda)/S_n(\lambda)$ increases in $\{\lambda\in\mathbb{R}^n\mid \lambda_i>0\}$, we see that the $(n-1,n-1)$-form in the last line is negative. Hence, the claim follows.
\end{proof}

Define the functional $\mathcal{J}$ by
\begin{equation*}
    \mathcal{J}(0)=0,\quad d\mathcal{J}(\varphi)(\psi):=\int_X\psi\ \left(\sum_{k=0}^{n-1}c_k\frac{\chi_{\varphi}^k\wedge\omega^{n-k}}{\chi_{\varphi}^n}-1\right)\chi_{\varphi}^n.
\end{equation*}

\begin{lem}\label{lem:convexalongflow}
    The functional $\mathcal{J}$ is convex along the mixed Hessian flow \eqref{eq:mixedHessianflowre}.
\end{lem}

\begin{proof}
    We can compute as
    \begin{align*}
        \frac{d^2}{dt^2}\mathcal{J}(\varphi_t)
        &= - \frac{d}{dt} \int_X \dot{\varphi}_t^2 \, \chi_t^n
        = - \int_X 2 \, \dot{\varphi}_t \, \ddot{\varphi_t} \, \chi^n_t 
          - \int_X \dot{\varphi}_t^2 \, n \, \delb \dot{\varphi_t}  \wedge \chi^{n-1}_t \\
        &= - \int_X 2 \, \dot{\varphi}_t \, \ddot{\varphi_t} \, \chi^n_t 
           + \int_X 2 \, \dot{\varphi}_t \, n \, \sqrt{-1} \partial \dot{\varphi}_t\wedge \bar{\partial} \dot{\varphi}_t \wedge \chi^{n-1}_t \\
        &= - \int_X 2\, \dot{\varphi}_t \, \sum_{k=0}^{n-1}c_k\left(-\frac{\delb \dot{\varphi_t}\wedge k\chi_t^{k-1}\wedge\omega^{n-k}}{\chi^n_t}+n\frac{\delb\dot{\varphi_t}\wedge\chi_t^{n-1}}{\chi_t^n}\right)\chi^n_t\\
        &\qquad- \int_X 2 \, \dot{\varphi}_t \, n \, \sqrt{-1} \partial \Lambda_{\chi_t} \omega \wedge \bar{\partial} \dot{\varphi}_t \wedge \chi^{n-1}_t \\
        &= -2 \sum_{k=0}^n c_k\int_X \dot{\varphi}_t \, F_{k,\chi_t,\omega}(\dot{\varphi}_t) \, \chi^n_t.
    \end{align*}
    Therefore, the claim follows from Lemma \ref{lem:F} and $c_k\ge0$ for $k\ge0$.
\end{proof}

\begin{lem}\label{lem:bound}
    For the mixed Hessian flow \eqref{eq:mixedHessianflowre}, there exists a constant $C$ such that
    $\vert\dot{\varphi_t}\vert\le C$.
\end{lem}

\begin{proof}
    By differentiating \eqref{eq:mixedHessianflowre} in $t$, we obtain
    $$\ddot{\varphi_t}=-\sum_{k=0}^{n-1}c_k\frac{\delb\dot{\varphi_t}\wedge \left(k\chi_t^{k-1}\wedge\omega^{n-k}-\frac{\chi_t^k\wedge\omega^{n-k}}{\chi_t^n}n\chi_t^{n-1}\right)}{\chi_t^n}.$$
    Since the $(n-1,n-1)$-form in the numerator on the right hand side is negative as noted in the proof of Lemma \ref{lem:F}, the standard maximum principle implies the desired result.
\end{proof}

\begin{lem}\label{lem:conv}
    Let $(X,\omega)$ be an $n$-dimensional compact Kähler manifold and $\chi$ a Kähler form. Suppose that we have
    $$\int_V\frac{n!}{p!}[\chi]^p-\sum_{k=1}^{n-1}c_k\frac{k!}{(k-n+p)!}[\chi]^{k-n+p}\wedge[\omega]^{n-k}\ge0$$
    for any $p$-dimensional subvariety $V$, where $p=1,2,\dots,{n-1}$. Then, the mixed Hessian flow $\chi_t$ satisfies $\Vert Q_{c_0,\omega}(\chi_t)-1 \Vert_{L^2(\omega)}\to0$ as $t\to\infty$.
\end{lem}

\begin{proof}
    First, note that the convexity proved in Lemma \ref{lem:conv} above implies that
    \begin{align*}
        \lim_{t\to\infty}\frac{\mathcal{J}(\varphi_t)}{t}=\lim_{t\to\infty}\frac{d}{dt}\mathcal{J}(\varphi_t)=-\lim_{t\to\infty}\Vert Q_{c_0,\omega}(\chi_t)-1\Vert^2_{L^2(\chi_t)}\le0.
    \end{align*}
    Denote the left hand side by $\mu$. We prove $\mu=0$. If this is true, we obtain the desired result. Indeed, by the bound of $\vert\dot{\varphi_t}\vert$ proved in Lemma \ref{lem:bound} above and the proof of \eqref{eq:masspositivity}, we have $\chi^n_t\ge c\omega^n$ for some $c>0$ uniformly in $t$. Thus, the desired result follows.
    
    For $\varepsilon>0$, define $\mathcal{J}_\varepsilon$ by
    $$\mathcal{J}_{\varepsilon}(\varphi):=\int_0^1\int_X\varphi\left(\sum_{k=1}^{n-1}c_k\frac{\chi_{t\varphi}^k\wedge\omega^{n-k}}{\chi_{t\varphi}^n}+\varepsilon\frac{\omega^n}{\chi_{t\varphi}^n}-1-a_{\varepsilon}\right)\chi_{t\varphi}^n\ dt,$$
    where $a_\varepsilon:=(\varepsilon\int_X\omega^n)/(\int_X\chi^n)>0$.
    Then, the critical point $\varphi_\varepsilon$ of $\mathcal{J}_{\varepsilon}$ satisfies
    \begin{equation}\label{eq:perturbgMA}
        \sum_{k=1}^{n-1}c_k\frac{\chi_{\varphi_\varepsilon}^k\wedge\omega^{n-k}}{\chi_{\varphi_\varepsilon}^n}+\varepsilon\frac{\omega^n}{\chi_{\varphi_\varepsilon}^n}=1+a_{\varepsilon}
    \end{equation}
    By assumption, we have
    $$\int_V(1+a_\varepsilon)\frac{n!}{p!}[\chi]^p-\sum_{k=1}^{n-1}c_k\frac{k!}{(k-n+p)!}[\chi]^{k-n+p}\wedge[\omega]^{n-k}>0$$
    for any $p$-dimensional subvariety $V$, where $p=1,2,\dots,{n-1}$.
    Therefore, by \cite{FM} (\cite{Datar-Pingali} if $X$ is projective), we obtain the solution of \eqref{eq:perturbgMA} with $I(\varphi_\varepsilon)=0$, where the functional $I$ is the Monge-Ampère energy, defined by
    $$I(0)=0, \quad dI(\varphi)(\psi)=\int_X\psi\,\chi_\varphi^n.$$
    By the proof of \cite[Theorem 3]{PT} or \cite[Theorem 7.5]{Sun15}, the flow
    \begin{equation}\label{eq:perturbmixedHessianflow}
    \begin{cases}
        \dot{\varphi_t}=1+a_\varepsilon-\sum_{k=0}^{n-1}c_k\frac{\chi_t^k\wedge\omega^{n-k}}{\chi_t^n}-\varepsilon\frac{\omega^n}{\chi^n_t},\\
        \varphi_t|_{t=0}=\varphi_0.
    \end{cases}
\end{equation}
    converges smoothly to $\varphi_\varepsilon+I(\varphi_0)$. Here, remark that the constant comes from the fact that $I$ is constant along the flow \eqref{eq:perturbmixedHessianflow}, since we have
    $dI(\varphi_t)(\dot{\varphi_t})=\int_X\dot{\varphi_t}\chi_t^n=0$
    by the definition of $a_\varepsilon$.
    Since \eqref{eq:perturbmixedHessianflow} is a negative gradient flow of $\mathcal{J}_{\varepsilon}$, we obtain $\mathcal{J}_{\varepsilon}\ge\mathcal{J}_{\varepsilon}(\varphi_\varepsilon).$ 
    On the other hand, we have
    \begin{equation*}
        \begin{aligned}
            \mathcal{J}_{\varepsilon}(\varphi_t)-\mathcal{J}_{\varepsilon}(\varphi_0)
            =&\int_0^t\int_X\dot{\varphi_t}\left(\sum_{k=1}^{n-1}c_k\frac{\chi_t^k\wedge\omega^{n-k}}{\chi_t^n}+\varepsilon\frac{\omega^n}{\chi_t^n}-1-a_{\varepsilon}\right)\chi_t^n\ dt,\\
            =&\mathcal{J}(\varphi_t)-\mathcal{J}(\varphi_0)+\int_0^t\int_X\dot{\varphi}_t\, \varepsilon\, \omega^n \,dt\\
            \le&\mathcal{J}(\varphi_t)-\mathcal{J}(\varphi_0)+C\varepsilon t,
        \end{aligned}
    \end{equation*}
    where the second equality follows from the fact that $I$ is constant along the flow \eqref{eq:mixedHessianflowre}, and the inequality follows from the bound of $\dot{\varphi_t}$ proved in Lemma \ref{lem:bound}. Therefore, we obtain
    $$\mathcal{J}(\varphi_t)\ge -C\varepsilon t-C_\varepsilon$$
    for some constants $C$ and $C_\varepsilon$.
    Dividing by $t$ and letting $t\to\infty$, we obtain
    $$\lim_{t\to\infty}\frac{\mathcal{J}(\varphi_t)}{t}\ge -C\varepsilon.$$
    Since $\varepsilon$ was arbitrary, we conclude that
    $$0\ge\mu:=-\lim_{t\to\infty}\Vert Q_{\omega}(\chi_t)-1\Vert^2_{L^2(\chi_t)}=\lim_{t\to\infty}\frac{\mathcal{J}(\varphi_t)}{t}\ge0,$$
    which is the desired claim.
\end{proof}

\begin{proof}[Proof of Theorem \ref{thm:Hessflow}]
    Take a subsequence $\{t_i\}$ of times such that $\chi_i:=\chi_{t_i}$ converges to $\chi_\infty$ in the sense of currents. Once one proves $\Vert Q_{c_0,\omega}(\chi_t)-1\Vert_{L^2(\omega)}\to0$ as $t\to\infty$, the same arguments as in \cite[Proposition 2.8]{Mur} yield $Q_{c_0,\omega}(\chi_\infty)\le_v1$. Here, we record it for the readers' convenience.
    With the same notations and the same arguments as \eqref{eq:Qconv}, noting that $c_0\ge0$ is constant by assumption, we have
    \begin{equation*}
        \begin{aligned}
            Q_{c_0,\omega_0}(\delb u_{i,\delta})
            &=Q_{c_0,I_n}\left(\omega_0^{-1}(z)\int_{\mathbb{C}^n}\delb u_i(z-\delta w)\rho(w)dw\right)\\
            &=Q_{c_0,I_n}\left(\int_{\mathbb{C}^n}(\omega_0^{-1}\chi_i)(z-\delta w)\rho(w)dw\right)\\
            &\le\int_{\mathbb{C}^n} \left(Q_{c_0,\omega}(\chi_i)(z-\delta w)-1\right)\rho(w)dw+1\\
            &\le\left(\int_{B_1} \left\vert Q_{c_0,\omega}(\chi_i)(z-\delta w)-1\right\vert^2dw\right)^{1/2}\Vert \rho\Vert_2+1\\
            &\le C_\delta \Vert Q_{c_0,\omega}(\chi_i)-1\Vert_{L^2(\omega)}+1,
        \end{aligned}
    \end{equation*}
    where the first inequality follows since $Q_{c_0,I_n}$ with $c_0\ge0$ is convex and increases in the set of positive Hermitian metrics. Hence, letting $i\to\infty, \ \delta\to0, \ \omega_0\to\omega$ in this order, by $\Vert Q_{c_0,\omega}(\chi_i)-1\Vert_{L^2(\omega)}\to0$, we conclude that $Q_{c_0,\omega}(\chi_\infty)\le_v1$. 
    
    Since $Q_{c_0,\omega}$ with $c_0\ge0$ increases and upper tests of psh functions are psh as in the proof of \cite[Proposition 1.3]{EGZ}, we see that $Q_{c_0,\omega}(\chi_\infty)\le_v1$ implies $P_\omega(\chi_\infty)\le_v1$. 
    Then, Lemma \ref{lem:visctopp} and Proposition \ref{prop:superineq} imply that $\chi_\infty$ is a solution of the gMA equation \eqref{eq:gMA}. The uniqueness of Theorem \ref{thm:main} implies that we do not need to take a subsequence and $\chi_t\to\chi_\infty$ as $t\to\infty$ in the sense of currents.
    \end{proof}

\section{The supercritical deformed Hermitian-Yang-Mills equation}\label{sec:dHYM}
In this section, we consider the dHYM equation using the same strategy. 

\subsection{Preliminaries}
For a Hermitian matrix $A\in\operatorname{Herm}(\mathbb{C}^n)$ and $\ell=1,\dots,{n-1}$, define
\begin{equation}\label{eq:thetaop}
    \tilde{\theta}^\ell(A):=\max_{i_1,\dots,i_\ell}\sum_{j\neq i_1,\dots,i_\ell} \operatorname{arccot}\lambda_j,\quad \theta(A):=\sum_j\operatorname{arccot}\lambda_j,
\end{equation}
where $\lambda_j$'s are eigenvalues of $A$. We particularly denote $\tilde{\theta}^1$ by $\tilde{\theta}.$
Fix constants $\theta,\Theta$ such that $0<\theta\le\Theta<\pi$ and define
\begin{equation}\label{eq:thetaset}
    \begin{aligned}
        &{\Gamma}_{\theta,\Theta}:=\{A\in\operatorname{Herm}(\mathbb{C}^n)\mid 0<\tilde{\theta}(A)<\theta,\quad 0<\theta(A)<\Theta\}.    \end{aligned}
\end{equation}
Denote its closure by $\bar{\Gamma}_{\theta,\Theta}$.
For $A\in\bar{\Gamma}_{\Theta,\Theta}$, we define
\begin{align*}
    P^\ell_{I_n}(A)= - \cot\tilde{\theta}^\ell(A),\quad Q_{c_{0,0},I_n}(A) = -\cot\theta(A)+\frac{c_{0,0}}{\operatorname{Im}\Pi_j(\lambda_j+\sqrt{-1})}.
\end{align*}
We particularly denote $P^1_{I_n}$ by $P_{I_n}$ and $Q_{0,I_n}$ by $Q_{I_n}$. Note that $\tilde{\theta}(A)<\theta$ is equivalent to
$P_{I_n}(A)<-\cot\theta$ and that $\theta(A)<\theta$ is equivalent to $Q_{I_n}(A)<-\cot\theta.$

\begin{prop}\label{prop:dhymproperties}
    Let $0<\theta\le\Theta<\pi$. Then, there exists a constant $\varepsilon_{3.1}>0$, depending on $\theta$, such that if a constant $c_{0,0}$ satisfies $c_{0,0}>-\varepsilon_{3.1}$, the following hold: 
    \begin{enumerate}[label={\upshape(\roman*)}]
        \item\label{item:dhymmonotonicity} The functions $\theta, \ \tilde{\theta},\ P^\ell_{I_n}$, and $Q_{c_{0,0},I_n}$ decrease.
        \item\label{item:dhympconv} $P^\ell_{I_n}$ is convex. 
        \item\label{item:dhymsublevel} The sets $\bar{\Gamma}_{\theta,\Theta}$ and $\{Q_{c_{0,0},I_n}\le -\cot\theta\}\cap\bar{\Gamma}_{\theta,\Theta}$ are convex.
        \item\label{item:dhymImest} There exists a constant $C_{3.1}$ depending on $\Theta$ such that for any $A\in\bar{\Gamma}_{\theta,\Theta}$, we have
        $$\frac{1}{\operatorname{Im}(\lambda_i+\sqrt{-1)^n}}\le\frac{1}{C_{3.1}},$$
        where $\lambda_i$ denotes the eigenvalues of $A$.
    \end{enumerate}
\end{prop}

\begin{proof}
    The first three of \ref{item:dhymmonotonicity} are obvious, while the last part is proved in \cite[Lemma 5.6 (3)]{GChen}. \ref{item:dhympconv} is proved in \cite[Proposition 2.2]{CLT}. The convexity of $\bar{\Gamma}_{\theta,\Theta}$ is proved in \cite[Lemma 5.6 (10)]{GChen}. For the convexity of the second item in \ref{item:dhymsublevel}, 
    note that the sublevel set can be rewritten as
    \begin{equation}\label{eq:sublevelset}
        \begin{aligned}
            &\{\lambda\in\mathbb{R}^n\mid Q_{c_{0,0},I_n}(\lambda)\le c\}\cap\bar{\Gamma}_{\theta,\Theta}\\
        =&\{\lambda\in\mathbb{R}^n\mid \operatorname{Re}(\lambda+\sqrt{-1})^n+c\operatorname{Im}(\lambda+\sqrt{-1})^n-c_{0,0}\ge0\}\cap\bar{\Gamma}_{\theta,\Theta}
        \end{aligned}
    \end{equation}
    where, we denote by the same notations $Q_{c_{0,0},I_n}$ and $\bar{\Gamma}_{\theta,\Theta}$ the operator and the subset in $\mathbb{R}^n$ defined in a trivial way from \eqref{eq:thetaop} and \eqref{eq:thetaset}. The convexity of the first term of the right-hand side of \eqref{eq:sublevelset} follows from \cite[Theorem 1.1]{Lina} if $c_{0,0}\ge-\varepsilon_{3.1}$ for some $\varepsilon_{3.1}>0$ depending on $\theta$. Indeed, by \cite[Lemma 4.2]{Lina}, if $c_{0,0}=0$, then the corresponding polynomial
    $$f(x):=\operatorname{Re}(x+\sqrt{-1})^n-\cot\theta\operatorname{Im}(x+\sqrt{-1})^n$$
    is strictly right-Noetherian. Since changing the zero-th order term $c_{0,0}$ does not affect any derivatives of $f(x)$ and the 1-st derivative of $f(x)$ is positive in $[x_0,\infty)$, where $x_0$ is the largest real root of $f(x)$, we can see that the polynomial after perturbing $c_{0,0}>-\epsilon_{3.1}$ is still strictly right-Noetherian. Hence, \cite[Theorem 1.1]{Lina} is applied. Therefore, \eqref{eq:sublevelset} is convex.
    Finally, to see that the same convexity occurs in the matrix level, recall the Ky-Fan inequality 
    $\lambda(tA+(1-t)B)\prec t\lambda(A)+(1-t)\lambda(B)$.
    Since the function $\mathbbm{1}_{\{Q_{c_{0,0},I_n}(\lambda)>-\cot\theta\}}$ is convex and symmetric, it is Schur-convex, that is, 
    $$\mathbbm{1}_{\{Q_{c_{0,0},I_n}(\lambda)>-\cot\theta\}}(\lambda)\le\mathbbm{1}_{\{Q_{c_{0,0},I_n}(\lambda)>-\cot\theta\}}(\lambda') \text{ if }\lambda\prec\lambda'.$$
    Hence, the desired result follows. 
    \ref{item:dhymImest} is proved in \cite[Lemma 5.6 (1)]{GChen}.
\end{proof}

For an $n$-dimensional Kähler manifold $(X,\omega)$ and a closed real $(1,1)$-form $\alpha$, we say $\alpha\in{\Gamma}_{\omega,\theta,\Theta}$ (resp. $\alpha\in\bar{\Gamma}_{\omega,\theta,\Theta}$)  if $\omega^{-1}\alpha\in{\Gamma}_{\theta,\Theta}$ (resp. $\omega^{-1}\alpha\in\bar{\Gamma}_{\theta,\Theta}$) at any point. We also define
$\theta_\omega,\ \tilde{\theta}_\omega,\ P^\ell_\omega$ and $Q_{c_0,\omega}$ in the same way as \eqref{eq:PQ}. We particularly denote $P^1_\omega$ by $P_\omega$ and $Q_{0,\omega}$ by $Q_\omega$. The twisted dHYM equation is given by
\begin{equation}\label{eq:tdHYMop}
    Q_{c_0,\omega}(\alpha)=-\cot\theta.
\end{equation}
A closed real $(1,1)$-form $\alpha$ is called a supercritical $\mathcal{C}$-subsolution to the twisted dHYM equation if $P_\omega(\alpha)<-\cot\theta$. Here, remark that we implicitly infer that $\alpha\in\bar{\Gamma}_{\omega,\Theta,\Theta}$ since we define $P_\omega$ on $\bar{\Gamma}_{\omega,\Theta,\Theta}$.
As in Definition \ref{defn:vsub}, we define a degenerate supercritical $\mathcal{C}$-subsolution and a viscosity subsolution as follows:

\begin{defn}\label{defn:dhymvsub}
    Let $(X,\omega)$ be a K{\"a}hler manifold. 
    We say that a closed real $(1,1)$-current $\alpha$ is a degenerate supercritical $\mathcal{C}$-subsolution (resp. a viscosity subsolution), denoted by $P_\omega(\alpha)\le_v-\cot\theta$ (resp. $P_\omega(\alpha)\le_v-\cot\theta$ and $Q_{c_0,\omega}(\alpha)\le_v-\cot\theta$), if for any point $x\in X$ and any coordinate neighborhood $U$ of $x$, we have  
    \begin{align*}
        &P_\omega(\delb q)(x)\le -\cot\theta, \\
        &\bigg( \text{resp. } P_\omega(\delb q)(x)\le -\cot\theta \text{ and } Q_{c_0,\omega}(\delb q)\le-\cot\theta,\bigg)
    \end{align*}
    for any upper test $q$ of $\varphi$ at $x \in U$, where $\varphi$ is an upper semicontinuous local potential of $\alpha$. Here, remark that we implicitly infer that $\delb q\in\bar{\Gamma}_{\omega,\Theta,\Theta}$ for any upper test $q$ since $P_\omega$ and $Q_{c_0,\omega}$ are defined on $\bar{\Gamma}_{\omega,\Theta,\Theta}$. 
\end{defn}

\begin{remark}\label{rem:dhymvisc}
    \begin{enumerate}[label={\upshape(\roman*)}]
        \item A smooth closed real $(1,1)$-form $\alpha$ is called a solution of the twisted dHYM equation if
        \begin{equation*}
            Q_{c_0,\omega}(\alpha)=-\cot\theta,\ P_\omega(\alpha)<-\cot\theta.
        \end{equation*}
        In this sense, a solution is a viscosity subsolution by Proposition \ref{prop:dhymproperties} \ref{item:dhymmonotonicity}. Here, note that any upper test $q$ satisfies $\delb q\in\bar{\Gamma}_{\omega,\Theta,\Theta}$.
        \item We will see that a closed real $(1,1)$-current $\alpha$ is a degenerate supercritical $\mathcal{C}$-subsolution if and only if $\alpha\in\bar{\Gamma}_{\omega,\theta,\Theta}$, where the definition of $\alpha\in\bar{\Gamma}_{\omega,\theta,\Theta}$ for a current $\alpha$ will be given in Definition \ref{def:dhymgammabar}. These are also equivalent to the condition $T^p_\omega(\alpha)\ge0$ for $p=1,\dots,n-1$, where $(p,p)$-current $T^p_\omega$ will be defined in Definition \ref{def:dhymllcurrent}. 
        \item\label{item:dhymmonovisc} Since $0<\operatorname{arccot}x<\pi$, if $P_\omega(\alpha)\le_v1$, then $P^\ell_\omega(\alpha)\le_v1$ for any $\ell=1,\dots,n-1$.
    \end{enumerate}
\end{remark}

\subsection{Existence}
In this subsection, we prove the existence part of Theorem \ref{thm:maindHYM} by the same strategy in subsection \ref{sec:existence}. More precisely, we prove the following theorem on the existence of a weak solution of the twisted dHYM equation, which is needed to prove the uniqueness later in subsection \ref{subsec:dhymuniqueness} as in subsection \ref{subsec:uniqueness}.

\begin{thm}\label{thm:tdHYM}
    Assume Setup \ref{setup2} and Assumption \ref{asmp:bdrydHYM}. Let $\Theta$ be a constant such that $0<\theta<\Theta<\pi$. Let $c_0$ be a bounded lower semicontinuous function such that 
    \begin{equation}
        c_0>-\varepsilon_{3.2} \ \text{and} \  \int_X c_0\omega^n=0,
    \end{equation} 
    where $\varepsilon_{3.2}$ depends on $\varepsilon_{3.1}$ as in Proposition \ref{prop:dhymproperties} and $\Theta$ as in \cite[Lemma 5.6 (5)]{GChen}. Then, there exists a unique quasi-psh function $\psi$ with $\sup\psi=0$ such that
    \begin{equation}\label{eq:tdHYM}
        \begin{cases}
            \mathrm{Re}\langle(\alpha_\psi+\sqrt{-1}\omega)^n\rangle-\cot\theta\mathrm{Im}\langle(\alpha_\psi+\sqrt{-1}\omega)^n\rangle-c_0\omega^n=0,\\
            \alpha_\psi\in\bar{\Gamma}_{\omega,\theta,\Theta},
        \end{cases}
    \end{equation}
    where $\langle\cdot\rangle$ denotes the nonpluripolar product and the definition of $\bar{\Gamma}_{\omega,\theta,\Theta}$ is given in Definition \ref{def:dhymgammabar}.
\end{thm}

In Setup \ref{setup2} and Assumption \ref{asmp:bdrydHYM}, by \cite[Theorem 1.3]{CLT}, there exists a (smooth) solution $\psi_i'$ of the dHYM equation 
\begin{equation*}
    Q_{\omega_i}(\alpha_{i,\psi'_i})=-\cot\theta_i.
\end{equation*}
Then, since $\alpha_{i,\psi'_i}\in{\Gamma}_{\omega_i,\theta_i,\Theta}$, by the same continuity method as in the proof of \cite[Proposition 5.5]{GChen} with the a priori estimates estabilished in \cite[Theorem 1.2]{Linb}, 
we obtain the solution $\psi_i$ of the twisted dHYM equation
\begin{equation}\label{eq:appdhym}
    \begin{cases}
        \operatorname{Re}(\alpha_{i,\psi_i}+\sqrt{-1}\omega_i)^n-c_{0,i}\omega_i^n=\cot\theta_i\operatorname{Im}(\alpha_{i,\psi_i}+\sqrt{-1}\omega_i)^n,\\
        \alpha_{i,\psi_i}\in{\Gamma}_{\omega_i,\theta_i,\Theta},\\
        \sup \psi_i=0,
    \end{cases}
    \end{equation}
    where $c_{0,i}$ is a smooth function  such that
    \begin{equation*}
        c_{0,i}>-\varepsilon_{3.2}, \  \lim_{i\to\infty} c_{0,i}=c_0,\int_X c_{0,i}\omega_i^n=0.
    \end{equation*}
Since $\alpha_{i,\psi_i}\in{\Gamma}_{\omega,\theta_i,\Theta}$, we have $\alpha_{i,\psi_i}-\cot\theta_i\omega_i>0$. Hence, by taking a subsequence, we obtain an $L^1$-limit $\psi_\infty$ of $\psi_i$. Define $\alpha_\infty=\alpha+\delb \psi_\infty$.

\begin{prop}\label{prop:vsubdHYM}
    We have $P_\omega(\alpha_\infty)\le_v -\cot\theta$ and $Q_{c_0,\omega}(\alpha_\infty)\le_v -\cot\theta$ in $\{\omega>0\}$.
\end{prop}

\begin{proof}
    We only point out the differences from the proof of Proposition \ref{prop:vsub}. To do the convolution arguments as in the proof of Proposition \ref{prop:vsub}, recall that, by \cite[Lemma 5.1]{CL}, for any sufficiently small $\varepsilon>0$, there exists $\sigma_0(n,\varepsilon,\Theta,c_0)$ such that the following holds: let $\alpha$ and $\omega$ be smooth closed real $(1,1)$-forms, where $\omega$ is Kähler. Suppose $0<\theta_\omega(\alpha)\le\Theta<\pi$ and $0<\tilde{\theta}_\omega(\alpha)\le\Theta<\pi$.
    If $\omega_0$ is a Kähler form with constant coefficients on a coordinate neighborhood satisfying $\vert\omega-\omega_0\vert_{\omega_0}<\sigma_0$, then
    \begin{equation}\label{eq:CL}
        \theta_{\omega_0}(\alpha)\le \theta_\omega(\alpha)+\varepsilon,\quad \tilde{\theta}_{\omega_0}(\alpha)\le \tilde{\theta}_\omega(\alpha)+\varepsilon.
    \end{equation}
    Moreover, we have
    \begin{equation}\label{eq:c_0term}
        \begin{aligned}
        &\left\vert\frac{c_0\omega_0^n}{\operatorname{Im}(\alpha+\sqrt{-1}\omega_0)^n}-\frac{c_0\omega^n}{\operatorname{Im}(\alpha+\sqrt{-1}\omega)^n}\right\vert\\
        \le&\vert c_0\vert \left(\left\vert\frac{\omega_0^n}{\operatorname{Im}(\alpha+\sqrt{-1}\omega_0)^n}-\frac{\omega_0^n}{\operatorname{Im}(\alpha+\sqrt{-1}\omega)^n}\right\vert+\frac{\vert\omega_0^n-\omega^n\vert}{\mathrm{Im}(\alpha+\sqrt{-1}\omega)^n}\right)\\
        \le&\vert c_0\vert \left(\frac{\omega_0^n}{\operatorname{Im}(\alpha+\sqrt{-1}\omega)^n}\left\vert\frac{\operatorname{Im}(\alpha+\sqrt{-1}\omega)^n}{\operatorname{Im}(\alpha+\sqrt{-1}\omega_0)^n}-1\right\vert+\frac{n\sigma_0}{(1-\sigma_0)^n}\frac{\omega^n}{\mathrm{Im}(\alpha+\sqrt{-1}\omega)^n}\right)\\
        \le&\frac{\vert c_0\vert}{(1-\sigma_0)^n C_{3.1}}\left(\left\vert\sum_{k=1}^n{n\choose k}\frac{\operatorname{Im}\big((\alpha+\sqrt{-1}\omega_0)^{n-k}\sqrt{-1}^k\big)\wedge(\omega-\omega_0)^k}{\operatorname{Im}(\alpha+\sqrt{-1}\omega_0)^n}\right\vert+n\sigma_0\right)\\
        \le&\varepsilon,
    \end{aligned}
    \end{equation}
    where the third inequality follows from Proposition \ref{prop:dhymproperties} \ref{item:dhymImest} and the last inequality follows from
    \begin{align*}
        \left\vert\frac{\operatorname{Im}\big((\alpha+\sqrt{-1}\omega_0)^{n-k}\sqrt{-1}^k\big)\wedge(\omega-\omega_0)^k}{\operatorname{Im}(\alpha+\sqrt{-1}\omega_0)^n}\right\vert\le C_{3.2}\vert(\omega-\omega_0)^k\vert_{\omega_0}\le\varepsilon'.
    \end{align*}
    for some constant $C_{3.2}$ and by taking $\sigma_0$ small enough. Here, the first inequality follows by the proof of \cite[Lemma 5.6 (1)]{GChen}, since $\theta_{\omega_0}(\alpha)\le\Theta+\varepsilon<\pi$ by \eqref{eq:CL}.
    Since $\alpha_{i,\psi_i}\in\Gamma_{\omega_i,\theta_i,\Theta}$, if we take a constant $c_{0,0}$ such that $c_{0,i}\ge c_{0,0}>-\varepsilon_{3.2}$ for any large $i$, and a local Kähler form $\omega_0$ with constant coefficients such that $\vert\omega_0-\omega_i\vert<\sigma_0$ for any large $i$, we can apply \eqref{eq:CL} and \eqref{eq:c_0term} to $\alpha_{i,\psi_i}$ and obtain 
    \begin{equation}\label{eq:deltamono}
        \begin{aligned}
            &{\theta}_{\omega_0}(\alpha_{i,\psi_i})(z-\delta w)\le {\theta}_{\omega_i}(\alpha_{i,\psi_i})(z-\delta w)+\varepsilon\le\Theta+\varepsilon,\\
            &\tilde{\theta}_{\omega_0}(\alpha_{i,\psi_i})(z-\delta w)\le \tilde{\theta}_{\omega_i}(\alpha_{i,\psi_i})(z-\delta w)+\varepsilon\le\theta_i+\varepsilon,\\
            &Q_{c_{0,0},\omega_0}(\alpha_{i,\psi_i})(z-\delta w)\le Q_{c_{0,0},\omega_i}(\alpha_{i,\psi_i})(z-\delta w)+\varepsilon\le-\cot\theta_i+\varepsilon.
        \end{aligned}
    \end{equation}
    By the convexity of $\bar{\Gamma}_{\theta,\Theta}$ and of the sublevel sets of $Q_{c_{0,0},I_n}$ as in Proposition \ref{prop:dhymproperties}, these inequalities imply $\delb u_{i,\delta}\in\bar{\Gamma}_{\omega_0,\theta_i,\Theta}$ and
    \begin{equation*}\label{eq:dhymQconv}
        \begin{split}
        Q_{c_{0,0},\omega_0}(\delb u_{i,\delta})(z)
        &=Q_{c_{0,0},I_n}\left(\omega^{-1}_0(z) \int_{\mathbb{C}^n}\delb u_{i}(z-\delta w)\rho(w)dw
        \right)\\
        &=Q_{c_{0,0},I_n}\left( \int_{\mathbb{C}^n}(\omega^{-1}_0 \alpha_{i,\psi_i})(z-\delta w)\rho(w)dw
        \right)\\
        &\le -\cot\theta_i+\varepsilon
        \end{split}
    \end{equation*}
for all $i$.
Noting that $\alpha_{i,\psi_i}\ge-\cot\theta_i\omega_i$ implies $\alpha_\infty\ge-\cot\theta\ \omega$, we see that $u_{\infty,\delta}$ satisfies the conditions in Proposition \ref{propCIL} as $\delta\to0$. 
Thus, as in the proof of \ref{prop:equivalence}, any upper test of $u_\infty$ can be approximated by the upper tests of $u_{i,\delta}$.
The rest of the proof is the same as that of Proposition \ref{prop:equivalence}.
\end{proof}

As in the previous section, with further arguments we can interpret conditions such as ``$P^\ell_\omega(\alpha_\varphi)\le_v-\cot
\theta$'' and ``$Q_{c_0,\omega}(\alpha_\varphi)\le_v-\cot\theta$'' to conditions involving convolution.

\begin{defn}\label{def:dhymgammabar}
    \begin{enumerate}[label={\upshape(\roman*)}]
        \item A closed real $(1,1)$-current $\alpha_\varphi$ is in $\bar{\Gamma}^\ell_{\omega,\theta,\Theta}$ if for any $\varepsilon>0$, for any open subset $U_0$ in any coordinate neighborhood, and for any Kähler form $\omega_0$ with constant coefficients in $U_0$ that is sufficiently close to $\omega$, we have
        \begin{equation}\label{eq:dhymconv}
            P^\ell_{\omega_0}(\delb u_\delta)\le-\cot\theta+\varepsilon, \quad Q_{\omega_0}(\delb u_\delta)\le-\cot\Theta+\varepsilon
        \end{equation}
        for any $\delta>0$, where $u$ is an upper semicontinuous local potential of $\alpha_\varphi$ and $u_\delta$ is its regularization via convolution. We particularly denote $\bar{\Gamma}^1_{\omega,\theta,\Theta}$ by $\bar{\Gamma}_{\omega,\theta,\Theta}$.
        \item A closed real $(1,1)$-current $\alpha_\varphi$ is in $\bar{\Gamma}_\omega'$ if $\alpha_\varphi\in\bar{\Gamma}_\omega$ and 
        \begin{align*}
            Q_{c_{0,0},\omega_0}(\delb u_\delta)\le -\cot\theta+\varepsilon, 
        \end{align*}
        for any constant $c_{0,0}$ such that $c_0\ge c_{0,0}>-\varepsilon_{3.2}$ on $U_0$, where the notations are the same as \eqref{eq:dhymconv}.
    \end{enumerate}
\end{defn}

\begin{prop}\label{prop:dhymequivalence}
    \begin{enumerate}[label={\upshape(\roman*)}]
        \item A closed real $(1,1)$-current $\alpha_\varphi$  satisfies $P^\ell_\omega(\alpha_\varphi)\le_v -\cot\theta$ and $Q_\omega(\alpha_\varphi)\le_v -\cot\Theta$ in $\{\omega>0\}$ if and only if $\alpha_\varphi\in\bar{\Gamma}^\ell_{\omega,\theta,\Theta}$ in $\{\omega>0\}$ and $\varphi$ is quasi-psh.
        \item Suppose that $P_\omega(\alpha_\varphi)\le_v -\cot\theta$ and $Q_\omega(\alpha_\varphi)\le_v -\cot\Theta$. Then, a closed real $(1,1)$-current $\alpha_\varphi$ satisfies $Q_{c_0,\omega}(\alpha_\varphi)\le_v-\cot\theta$ in $\{\omega>0\}$ if and only if $\alpha_\varphi\in\bar{\Gamma}'_\omega$ in $\{\omega>0\}$.
    \end{enumerate}
\end{prop}

\begin{proof}
     First, we prove the ``if'' parts. 
     Since $\varphi$ is quasi-psh, a sequence $\{u_\delta\}$, where $u_\delta$ is the regularization of an upper semicontinous local potential of $\alpha_\varphi$, satisfies the conditions in Proposition \ref{propCIL}. Hence, the arguments as before imply the desired conclusions.
     Next, we prove the ``only if'' parts. Note that the assumptions imply that $P^{n-1}_{\omega}(\alpha_\varphi)\le_v-\cot\theta$, which yields $\alpha_\varphi-\cot\theta\, \omega\ge0$ as a current (see, e.g., the proof of \cite[Proposition 1.3]{EGZ}). Hence, a function $u^{C,\varepsilon}$ is quasi-psh, where the notatinos are given in the same manner as in the proof of Proposition \ref{prop:equivalence}. Using \eqref{eq:deltamono} (and its variant for $\tilde{\theta}^\ell$, which can be proved in the same as \cite[Lemma 5.1]{CL}) and the fact that $u^{C,\varepsilon}$ is quasi-psh, we can make arguments similar to \eqref{eq:supconv}.
     The rest of the proof is the same as that of Proposition \ref{prop:equivalence}.
\end{proof}

Next, we claim that a viscosity subsolution is a pluripotential subsolution.

\begin{defn}\label{def:dhymllcurrent}
    For a closed real $(1,1)$-current $\alpha_\varphi\in\bar{\Gamma}_{\omega,\Theta,\Theta}$, we define $(p,p)$-current $T^p_\omega(\alpha_\varphi)$ by
    $$T^p_\omega(\alpha_\varphi)=\operatorname{Re}\langle(\alpha_\varphi+\sqrt{-1}\omega)^p\rangle-\cot\theta\operatorname{Im}\langle(\alpha_\varphi+\sqrt{-1}\omega)\rangle^p,$$
    where $\langle\cdot\rangle$ denotes the nonpluripolar product, which is well-defined since $\varphi$ is quasi-psh.
\end{defn}

\begin{lem}\label{lem:dhymvisctopp} Suppose $\alpha_\varphi\in\bar{\Gamma}_{\omega,\Theta,\Theta}$.
    \begin{enumerate}[label={\upshape(\roman*)}]
        \item\label{item:dhymvisctoppn-1} 
        If $P^\ell_{\omega}(\alpha_\varphi)\le_v -\cot\theta$ in $\{\omega>0\}$, then $T^p_\omega(\alpha_\varphi)\ge0$ in $X$ for all $p\le n-\ell$.
        \item\label{item:dhymvisctoppn} 
        Suppose $P_{\omega}(\alpha_\varphi)\le_v -\cot\theta$. If $Q_{c_0,\omega}(\alpha_\varphi)\le_v -\cot\theta$ in $\{\omega>0\}$, then $T^n_\omega(\alpha_\varphi)-c_0\omega^n\ge0$ in $X$. 
    \end{enumerate}
\end{lem}

\begin{proof}
    We only point out the differences from the proof of Lemma \ref{lem:visctopp}.
    We prove \ref{item:dhymvisctoppn-1} here. The proof of \ref{item:dhymvisctoppn} is similar to that of \ref{item:dhymvisctoppn-1}, so we omit it. By Proposition \ref{prop:dhymequivalence}, we obtain $P^\ell_{\omega_0}(\delb (u^C)_\delta)\le-\cot\theta+\varepsilon$. Here, a function $u^C$ is defined as $u^C:=\sup \{u,av-C\},$ where $v$ is a local upper semicontinuous potential of $\omega_0$ and $a>0$ is large enough to satisfy $a\omega\in\bar{\Gamma}_{\omega,\theta,\Theta}$.
    Note that $\operatorname{Im}(\delb (u^C)_\delta+\sqrt{-1}\omega_0)^\ell\ge0$ since $\delb (u^C)_\delta\in\bar{\Gamma}_{\omega_0,\Theta,\Theta}$ as in the proof of Proposition \ref{prop:vsubdHYM}. Hence, the condition is equivalent to $T^{n-\ell}_{\varepsilon,\omega_0}(\delb (u^C)_\delta)\ge0$, where $T^{n-\ell}_{\varepsilon,\omega_0}$ is defined by replacing $\cot\theta$ with $\cot\theta-\varepsilon$ in Definition \ref{def:dhymllcurrent}. Since $(u^C)$ is quasi-psh, a sequence $\{(u^C+F)_\delta\}_\delta$ decreases as $\delta\to0$, where $F$ is a local potential of a Kähler form such that $\delb (u^C+F)\ge0$.
    Hence, we have
    $$T^{n-\ell}_{\varepsilon,\omega_0}(\delb ((u^C+F)_\delta-F_\delta))\to T^{n-\ell}_{\varepsilon,\omega_0}(\delb u^C)\ge0 \ \text{ as } \delta\to0,$$
    where the convergence follows from \cite[Theorem 2.1]{BT}. The rest of the proof is the same as that of Lemma \ref{lem:visctopp}. 
\end{proof}

Finally, we prove a mass-type inequality to see that a subsolution is a solution.

\begin{prop}\label{prop:dhymsuperineq}
    If a current $\alpha_\varphi$ satisfies $P_\omega(\alpha_\varphi)\le_v-\cot\theta$ in $\{\omega>0\}$, then we have
    $$\int_XT^n_\omega(\alpha_\varphi)-c_0\omega^n=\int_XT^n_\omega(\alpha_\varphi)\le0.$$
\end{prop}

\begin{proof}
    We only point out the differences from the proof of Proposition \ref{prop:superineq}. We define $\varphi_{i,s}:=(1-s)\varphi+s\psi_i$, where $\psi_i$ is a solution of \eqref{eq:appdhym}. Then, we can see that for any sufficiently small $\varepsilon>0$, we have \begin{equation}\label{eq:phiis}
        P_\omega(\alpha_{i,\varphi_{i,s}})\le-\cot\theta_i+\varepsilon
    \end{equation} in $\{\omega>0\}$ for sufficiently large $i$. Indeed, since $\omega_i\to\omega$, using \eqref{eq:CL} we obtain $P_\omega(\alpha_{i,\psi_i})\le-\cot(\theta_i+\varepsilon)$. Also, since $\alpha_i\ge\alpha$ and $-\cot\theta\le-\cot\theta_i$ as in Assumption \ref{asmp:bdrydHYM}, we have $P_\omega(\alpha_{i,\varphi})\le_v-\cot\theta_i+\varepsilon$. Thus, by the convexity of $P$ as in Proposition \ref{prop:dhymproperties} \ref{item:dhympconv}, we see \eqref{eq:phiis}. By Lemma \ref{lem:dhymvisctopp}, we conclude $T^{n-1}_{i,\omega}(\alpha_{i,\varphi_{i,s}})$ is positive, where $T^p_{i,\omega}$ is defined by replacing $\cot\theta$ with $\cot\theta_i+\varepsilon$.
    The rest of the proof is the same as that of Proposition \ref{prop:superineq}, once we prove the lemma below, which corresponds to Lemma \ref{lem:ppineq}.
\end{proof}

\begin{lem}
    Let $\psi$ be a quasi-psh function such that $T^{n-1}_{i,\omega}(\alpha_{i,\psi})$ is positive. If quasi-psh functions $\varphi_1,\varphi_2$ satisfy $\varphi_1\le\varphi_2+C$ for some constant $C$, then we have 
    $$\int_X\big\langle\alpha_{i,\varphi_1}\wedge T^{n-1}_{i,\omega}(\alpha_{i,\psi})\big\rangle\le\int_X\big\langle\alpha_{i,\varphi_2}\wedge T^{n-1}_{i,\omega}(\alpha_{i,\psi})\big\rangle.$$
\end{lem}

\begin{proof}
    The proof is completely the same as that of Lemma \ref{lem:ppineq} noting that $\psi,\varphi_1,\varphi_2$ are quasi-psh.
\end{proof}

We summarize the proof here.
\begin{proof}[Proof of Theorem \ref{thm:tdHYM}]
    By Setup \ref{setup2} and Assumption \ref{asmp:bdrydHYM}, we have the solutions $\psi_i$ of approximate twisted dHYM equations \eqref{eq:appdhym}. By taking a subsequence, we obtain a weak limit $\psi_\infty$ of $\{\psi_i\}$. By Proposition \ref{prop:vsubdHYM} and Lemma \ref{lem:dhymvisctopp}, a current $\alpha_\infty:=\alpha+\delb\psi_\infty$ satisfies $P_\omega(\alpha_\infty)\le_v-\cot\theta$, $Q_\omega(\alpha_\infty)\le_v-\cot\Theta$ in $\{\omega>0\}$ and $T^n_{\omega}(\alpha_\infty)-c_0\omega^n\ge0$ in $X$. The former two are equivalent to $\alpha_\infty\in\bar{\Gamma}_{\omega,\theta,\Theta}$ by Proposition \ref{prop:dhymequivalence}.
    By Proposition \ref{prop:dhymsuperineq}, the total mass of $T^n_\omega(\alpha_\infty)-c_0\omega^n$ equals zero. Thus, we obtain $T^n_\omega(\alpha_\infty)-c_0\omega^n=0$ and $\alpha_\infty\in\bar
    \Gamma_{\omega,\theta,\Theta}$, which are \eqref{eq:tdHYM}.
\end{proof}

\subsection{Uniqueness}\label{subsec:dhymuniqueness}

In this subsection, by the same strategy as in subsection \ref{subsec:uniqueness}, we prove the uniqueness part of Theorem \ref{thm:maindHYM}.

\begin{lem}\label{lem:dhymapp}
    Let $\Omega$ and $\Omega'$ be bounded domains in $\mathbb{C}^n$ such that $\Omega'\subset\subset\Omega$, and $\omega$ be a Kähler form in $\Omega$. \begin{enumerate}[label={\upshape(\roman*)}]
        \item\label{item:dhymappP} Suppose that $\varphi$ is a bounded function in $\Omega$ such that $P^\ell_\omega(\delb\varphi)\le_v-\cot\theta$ in $\Omega$. Then, for any $d>-\cot\theta$, there exist a sequence of smooth functions $\{\varphi_a\}_a$ and smooth psh functions $\{f_a\}$ in $\Omega'$ such that $\varphi_a+f_a$ is psh and decreases to $\varphi+f$, where $f$ is a smooth function in $\Omega'$, and that $P^\ell_\omega(\delb \varphi_a)\le d$ in $\Omega'$.
        \item\label{item:dhymappQ} Suppose that $\varphi$ is a bounded psh function in $\Omega$ such that $P_{\omega}(\delb \varphi)\le_v-\cot\theta$ and $Q_{c_0,\omega}(\delb \varphi)\le_v-\cot\theta$ in $\Omega$, where $c_0>-\varepsilon_{3.2}$ is a continuous function. Then, for any $d>-\cot\theta$, there exist a sequence of smooth functions $\{\varphi_a\}_a$ and smooth psh functions $\{f_a\}$ in $\Omega'$ such that $\varphi_a+f_a$ is psh and decreases to $\varphi+f$, where $f$ is a smooth function in $\Omega'$, and that $P_\omega(\delb \varphi_a)\le d$ and $Q_{c_0,\omega}(\delb \varphi_a)\le_v d$ in $\Omega'$.
    \end{enumerate}
\end{lem}

\begin{proof}
    First, we consider \ref{item:dhymappP}. As in the previous section, we prove this using the simple version of the gluing argument in 
    \cite[Section 4 and p596--p597]{GChen}, based on \cite{BK}. Fix $\varepsilon_0>0$ such that $-\cot\theta+\varepsilon_0< d$. For any $\varepsilon>0$, define
    $$\psi_\varepsilon:=\varphi+\varepsilon(\vert z\vert^2+C),$$
    where $C$ is a large constant defined later. Then, we see that $(-\varepsilon\omega_\mathrm{Euc})_{\psi_\varepsilon}\in\bar{\Gamma}^\ell_{\omega,\theta,\Theta}$ by Proposition \ref{prop:dhymequivalence}, where $\omega_\mathrm{Euc}:=\sum\sqrt{-1}dz^i\wedge d\bar{z}^i$. We cover $\Omega'$ with a finite number of coordinate balls $B^i_r$ centered at $x_i$ of radius $r\ll1$ (with respect to $\omega_\mathrm{Euc}$) so small that we have 
    \begin{equation}\label{eq:dhymconstomega}
        \omega_0^i-\sigma\le\omega\le\omega^i_0+\sigma \ \text{on} \ B^i_{2r}:=B_{2r}(x_i),
    \end{equation}
    where $\omega_0^i$ is Kähler form on $B^i_{2r}$ with constant coefficients. Here, the constant $\sigma$ is smaller than $\sigma_0(n,\varepsilon_1,\Theta,c_0)$, which is given above \eqref{eq:CL}, where $\varepsilon_1$ is a constant such that $-\cot\theta+\varepsilon_0+\varepsilon_1\le d$. Note that we have $\delb (\varphi+A\vert z\vert^2)\ge0$ by Proposition \ref{prop:dhymequivalence}.  
    Adapting the arguments of \cite[Proposition 4.3 (3)]{GChen}, for $\delta>0$ small enough depending on $A$, we obtain a sequence of smooth functions $\{\Psi_{\varepsilon,\delta}\}_\delta$, given by the regularized maximum of 
    \begin{equation}\label{eq:dhympsidelta}
        (\psi_\varepsilon-\varepsilon\vert z-x_i\vert^2)_\delta=(\psi_\varepsilon+A\vert z\vert^2-\varepsilon\vert z-x_i\vert^2)_\delta-(A\vert z\vert^2)_\delta
    \end{equation}
    running $i$, where $(\cdot)_\delta$ denotes regularization by convolution as before. 
    The functions $\Psi_{\varepsilon,\delta}$ satisfy
    $$P^\ell_{\omega_0}(\delb\Psi_{\varepsilon,\delta})\le -\cot\theta+\varepsilon_0.$$
    By the definition of the regularized maximum, we see that a sequence
    $\Psi_{\varepsilon,\delta}+(A\vert z\vert)_\delta$ decreases.
    Finally, by \eqref{eq:CL}, we have $P^\ell_\omega(\delb\Psi_{\varepsilon,\delta})\le-\cot\theta+\varepsilon_0+\varepsilon_1\le d$.
    
     The same argument yields \ref{item:dhymappQ}. Indeed, we first take a finer cover of $\Omega'$ so that on each $B^i_{2r}$ we have a constant $c_{0,0}^i$ such that
     \begin{equation}\label{eq:dhymc_00^i}
         -\varepsilon_{3.2}<c_{0,0}^i\le c_0\le (1+\varepsilon')c_{0,0}^i
     \end{equation}
     in addition to \eqref{eq:dhymconstomega}, where $\varepsilon'$ will be chosen later. Then, denoting by $\psi^i_{\varepsilon,\delta}$ the function in \eqref{eq:dhympsidelta}, from \eqref{eq:dhymc_00^i}, we obtain the inequality
     \begin{equation}
         \begin{aligned}
             Q_{\omega_0^i}(\delb \psi^i_{\varepsilon,\delta})+\frac{c_0}{1+\varepsilon'}\frac{(\omega_0^i)^n}{\mathrm{Im}(\delb\psi^i_{\varepsilon,\delta}+\sqrt{-1}\omega_0^i)^n}\le-\cot\theta+\varepsilon_0.
         \end{aligned}
     \end{equation}
     Then, by Proposition \ref{prop:dhymproperties} \ref{item:dhymImest}, \eqref{eq:CL}, \eqref{eq:c_0term}, and \eqref{eq:dhymconstomega}, we see that $$Q_{c_0,\omega}(\delb\psi_{\varepsilon,\delta}^i)\le-\cot\theta+\varepsilon_0+\varepsilon_1\le d.$$
     By the arguments in \cite[p561--p562]{GChen} with the convexity of the sublevel sets and the monotonicity of $Q_{c_0,\omega}$, we obtain the desired result.
\end{proof}

Using this approximation, we obtain an inequality analogous to Lemma \ref{lem:comp}.

\begin{lem}\label{lem:dhymcomp}
    Let $\Omega$ and $\Omega'$ be bounded domains in $\mathbb{C}^n$ such that $\Omega'\subset\subset\Omega$, and $\omega$ be a Kähler form in $\Omega$. Suppose that $u_0,u_1$ are bounded functions on $\Omega$ such that $P^\ell_\omega(\delb u_0)\le_v-\cot\theta$ and $P^\ell_\omega(\delb u_1)\le_v -\cot\theta$ on $\Omega$ and $\liminf_{z\to\partial\Omega'}(u_1-u_0)\ge0$. Then, we have
    \begin{equation}
        \int_{\{u_1<u_0\}\cap\Omega'}\omega^{n-p}\wedge T^p_\omega(\delb u_0)\le\int_{\{u_1<u_0\}\cap\Omega'}\omega^{n-p}\wedge T^p_\omega(\delb u_1),
    \end{equation}
    where $p:=n-\ell+1$. 
\end{lem}

\begin{proof}
    We use the same proof as in Lemma \ref{lem:comp}. 
    The only difference is that $u_i$ is not psh but quasi-psh. However, to prove inequality \eqref{eq:lincomp} with the same argument, it is sufficient to add $A\omega$ to the second term $\delb u_{0,a}$ on the left-hand side and to $\delb u_{1,b}$ on the right-hand side so that both become positive. Then, we let $a,b\to\infty$ and then let $d\to-\cot\theta$ to obtain inequality \eqref{eq:lincomp'}. Here, when applying the convergence argument of \cite[Theorem 4.1]{BT}, remark that we have
    \begin{align*}
        \omega^k\wedge(\delb u_{0,a})^{n-k}
        &=\omega^k\wedge(\delb (u_{0,a}+f_a)-\delb f_a)^{n-k},
    \end{align*}
    where $\{f_a\}$ is a sequence as in Lemma \ref{lem:dhymapp}. Thus, integration of $\omega^k\wedge(\delb u_{0,a})^{n-k}$ on a set $E$ can be estimated by $\mathrm{cap}(E)$. 
\end{proof}

\begin{cor}\label{cor:dhymcomp}
    Let $\Omega$ and $\Omega'$ be bounded domains in $\mathbb{C}^n$ such that $\Omega'\subset\subset\Omega$, and $\omega$ be a Kähler form in $\Omega$. Suppose that $u_0,u_1$ are bounded functions in $\Omega$ such that 
    \begin{equation*}
        \begin{cases}
            P^{\ell}_\omega(\delb u_i)\le_v1 \ \text{ in }\Omega \ \text{ for } i=1,2,\\
            T^p_\omega(\delb u_1)\le T^p_\omega(\delb u_0) \ \text{ in }\Omega',\\
            \liminf_{z\to\partial\Omega'} (u_1-u_0)\ge 0,
        \end{cases}
    \end{equation*}
    where $p=n-\ell+1$. Then, we have $u_0\le u_1$ in $\Omega'$.
\end{cor}

\begin{proof}
    The proof is completely the same as that of Corollary \ref{cor:comp}, once one notifies that the set $S$ as in the proof has a positive Lebesgue measure in the situation since $u_0$ and $u_1$ are quasi-subharmonic.
\end{proof}

\begin{prop}\label{prop:dhymviscpp}
    Suppose $\alpha_\varphi\in\bar{\Gamma}_{\omega,\Theta,\Theta}$.
    \begin{enumerate}[label={\upshape(\roman*)}]
        \item\label{item:dhymposin-1} 
        $P^\ell_{\omega}(\alpha_\varphi)\le_v 1$ in $\{\omega>0\}$ if and only if $T^p_\omega(\alpha_\varphi)\ge0$ in $X$ for all $p\le n-\ell$.
        \item\label{item:dhymposin} 
        Suppose $P_{\omega}(\alpha_\varphi)\le_v 1$. Then, $Q_{c_0,\omega}(\alpha_\varphi)\le_v 1$ in $\{\omega>0\}$ if and only if $T^n_\omega(\alpha_\varphi)-c_0\omega^n\ge0$ in $X$.
    \end{enumerate}
\end{prop}

\begin{proof}
    The proof is completely the same as that of Proposition \ref{prop:viscpp}.
\end{proof}

Now, we prove two more lemmas as before, to prove the uniqueness.

\begin{lem}
    Suppose that $\varphi_1$ and $\varphi_2$ are solutions of \eqref{eq:dHYM}. Then, for $i=1,2$, we have 
    $$\big\langle\alpha_{\varphi_1}\wedge T^{n-1}_\omega(\alpha_{\varphi_i})\big\rangle=\big\langle\alpha_{\varphi_2}\wedge T^{n-1}_\omega(\alpha_{\varphi_i})\big\rangle$$
\end{lem}

\begin{proof}
    The proof is completely the same as that of Lemma \ref{lem:uni}.
\end{proof}

\begin{lem}
    Let $\alpha_\infty:=\alpha+\delb\varphi_\infty$ be a solution of \eqref{eq:dHYM}. Then, for any function $\varphi$ such that $P_\omega(\alpha_\varphi)\le_v-\cot\theta$, we have
    $$\int_X\big\langle\alpha_\varphi\wedge T^{n-1}_\omega(\alpha_\infty)\big\rangle\le\int_X\big\langle\alpha_\infty\wedge T^{n-1}_\omega(\alpha_\infty
    ).$$
\end{lem}

\begin{proof}
    The proof is completely the same as that of Lemma \ref{lem:solineq}.
\end{proof}

\begin{proof}[Proof of the uniqueness part of Theorem \ref{thm:maindHYM}]
    We only point out the differences from the proof of Theorem \ref{thm:main}. Suppose that $\varphi_1$ and $\varphi_2$ are solutions of the dHYM equation \eqref{eq:dHYM}. We fix $i=1,2$.
    First, note that by the proof of \cite[Lemma 5.6 (1)]{GChen}, there exists a constant $c>0$ such that for a smooth $(1,1)$-form $\alpha\in\bar{\Gamma}_{\omega,\theta,\Theta}$, we have
    $$\omega\wedge\mathrm{Im}(\alpha+\sqrt{-1}\omega)^{n-1}\ge c\omega^n.$$
    Therefore, we have
    \begin{align*}
        &\big\langle(\alpha_{\varphi_i}-\cot\theta\omega)\wedge T^{n-1}_\omega(\alpha_{\varphi_i})\big\rangle\\
        =&\big\langle(\alpha_{\varphi_i}-\cot\theta\omega)\wedge \big(\mathrm{Re}(\alpha_{\varphi_i}+\sqrt{-1}\omega)^{n-1}-\cot\theta\mathrm{Im}(\alpha_{\varphi_i}+\sqrt{-1}\omega)^{n-1}\big)\big\rangle\\
        =&\big\langle\alpha_{\varphi_i}\wedge\mathrm{Re}(\alpha_{\varphi_i}+\sqrt{-1}\omega)^{n-1}-\cot\theta\mathrm{Im}(\alpha_{\varphi_i}+\sqrt{-1}\omega)^n+\cot^2\theta\omega\wedge\mathrm{Im}(\alpha_{\varphi_i}+\sqrt{-1}\omega)^{n-1}\big\rangle\\
        =&(\cot^2\theta+1)\omega\wedge\mathrm{Im}\big\langle(\alpha_\varphi+\sqrt{-1}\omega)^{n-1}\big\rangle\ge c\omega^n.
    \end{align*}
    After this observation, the rest of the proof is the same as in the previous section, since $\varphi_1$ and $\varphi_2$ are quasi-psh.
\end{proof}

\subsection{Weak convergence of the deformed Hermitian-Yang-Mills flow}

In this subsection, we prove Theorem \ref{thm:dHYMflow} by the same proof as in subsection \ref{subsec:mixedHessianflow}. Recall that the $\mathcal{J}$-functional in the context of the dHYM equation is defined by
\begin{equation*}
    \mathcal{J}(0)=0,\quad d\mathcal{J}(\varphi)(\psi)=\int_X\psi\ \mathrm{Im}\Big(e^{-\sqrt{-1}\theta}(\alpha_\varphi+\sqrt{-1}\omega)^n\Big).
\end{equation*}
The following lemma corresponds to Lemma \ref{lem:convexalongflow} and is proved in \cite[Lemma 3.1]{Mur}. 

\begin{lem}\label{lem:dhymconvexalongflow}
    The $\mathcal{J}$-functional is convex along the dHYM flow.
\end{lem}

 The following corresponds to Lemma \ref{lem:conv}.
 
\begin{lem}\label{lem:dhymconv}
    Let $(X,\omega)$ be an $n$-dimensional compact Kähler manifold and $\alpha$ be a closed real $(1,1)$-form. Assume Setup \ref{setup2} and Assumption \ref{asmp:bdrydHYM} with $\omega_i=\omega$ and $\alpha_i=\alpha$ for any $i$. Then, the dHYM flow satisfies $\Vert Q_\omega(\alpha_t)+\cot\theta\Vert_{L^2(\omega)}\to0$ as $t\to\infty$.
\end{lem}

\begin{proof}[Proof of Lemma \ref{lem:dhymconv}]
     The strategy is completely the same as that of Lemma \ref{lem:conv}. In the two-dimensional case, the lemma is proved in \cite[Proposition 3.3]{Mur}. In our high-dimensional case, for any $\varepsilon>0$, by \cite{Linb} and \cite[Proposition 5.5]{GChen}, Assumption \ref{asmp:bdrydHYM} implies the existence of a solution of the twisted dHYM equation:
\begin{equation*}
    \begin{cases}
        \mathrm{Re}(\alpha+\sqrt{-1}\omega)^n=\cot\theta\mathrm{Im}(\alpha+\sqrt{-1}\omega)^n+\varepsilon\omega^n,\\
        \alpha\in\Gamma_{\omega,\theta,\Theta}.
    \end{cases}
\end{equation*}
Then, by $0<\inf\theta_\omega(\alpha_0)\le\theta_\omega(\alpha_t)\le\sup\theta_\omega(\alpha_0)<\pi$ as in \cite[Lemma 3.2]{FYZ} and Lemma \ref{lem:dhymconvexalongflow}, the same proof as in \cite[Proposition 3.3]{Mur} implies the desired result.
\end{proof}

\begin{proof}[Proof of Theorem \ref{thm:dHYMflow}]
    As in the proof of Theorem \ref{thm:Hessflow}, once one establishes Lemma \ref{lem:dhymconv}, the arguments as in \eqref{eq:dhymQconv} yield $Q_\omega(\alpha_\infty)\le_v-\cot\theta$, where $\alpha_\infty$ is a subsequential limit of the dHYM flow. Then, Lemma \ref{lem:dhymvisctopp} and Proposition \ref{prop:dhymsuperineq} imply that $\alpha_\infty$ is a solution of the dHYM equation \eqref{eq:dHYM}. The uniqueness of Theorem \ref{thm:maindHYM} implies that we do not need to take a subsequence and $\alpha_t\to\alpha_\infty$ as $t\to\infty$ in the sense of currents.
\end{proof}

\end{document}